\documentclass[12pt]{article}
\usepackage{amsmath,amsthm,amssymb}

\newtheorem{theorem}{Theorem}
\newtheorem{proposition}{Proposition}
\newtheorem{lemma}{Lemma}
\newtheorem{definition}{Definition}
\newtheorem{corollary}{Corollary}

\newcommand{\ga}{\alpha}
\newcommand{\gb}{\beta}
\newcommand{\gc}{\gamma}
\newcommand{\gd}{\delta}
\newcommand{\gl}{\lambda}
\newcommand{\gs}{\sigma}
\newcommand{\hess}{\nabla^2}
\newcommand{\gra}{\nabla}
\newcommand{\de}{\partial}
\newcommand{\bpf}{\begin{proof}}
\newcommand{\epf}{\end{proof}}
\newcommand{\beq}{\begin{equation}}
\newcommand{\eeq}{\end{equation}}

\begin{document}
\title{Conformal Deformation on Manifolds \\
 with Boundary}
\vskip 1em
 \author{Szu-yu Sophie Chen
 \footnote{The author was supported in part by the Miller Institute for Basic Research in Science.}}
 \date{}
\maketitle
\begin{abstract}
We consider natural conformal invariants arising from the
Gauss-Bonnet formulas on manifolds with boundary, and study
conformal deformation problems associated to them.
 \end{abstract}

The purpose of this paper is to study 
 conformal deformation problems associated to conformal
invariants on manifolds with boundary.  From analysis point of view, the problem 
becomes a non-Dirichlet boundary value problems for fully nonlinear equations.
This may be  compared   to a work by Lieberman-Trudinger \cite{LieT86} on the oblique-type boundary value problems.

Let $(M, g)$ be a compact, connected Riemannian manifold of
dimension $n \geq 3$ with boundary $\de M$. We denote the
Riemannian curvature, Ricci curvature, scalar curvature, mean
curvature, and the second fundamental form by $Riem, Ric, R , h$,
and $L_{\alpha \beta},$ respectively.

 The Yamabe constant for compact manifolds with boundary is a conformal
  invariant, defined as
 $$Y(M, \de M, [g]) = \inf_{\hat{g} \in [g], V_{\hat g} = 1}  ( \int_M R_{\hat{g}} +
 \oint_{\partial M} h_{\hat{g}} ),$$
 where $[g]$ is the conformal class of $g.$ 
  It was proved by Escobar \cite{Es92} that for
most compact manifolds with boundary, the Yamabe problem is solvable; i.e., there exists
a conformal metric  such that the scalar
curvature is constant and the mean curvature is zero.

To study a nonlinear version of the Yamabe problem, we consider
the Schouten tensor defined as
$$ A_g =\frac{1}{n-2} ( Ric - \frac{R}{2(n-1)} g ) .$$
 The problem  consists in
 finding a
 metric $\hat{g}= e^{-2u} g$  such that the $\gs_k(A_{\hat {g}})$ curvature
 is constant, where $\gs_k$ is the $k$th elementary symmetric
 function of the eigenvalues of $A_{\hat{g}}.$ When $k= 1,$
  the problem reduces to the original
 Yamabe problem.

 In dimension four, the $\gs_2 (A_g)$ curvature is related
 to the Gauss-Bonnet formula
  and  $\int_M \gs_2 (A_g)$ is a conformal
invariant on closed manifolds. Chang-Gursky-Yang
\cite{CGY02}, \cite{CGY02a} proved that if the Yamabe constant
$Y(M, [g])$ and $\int_M \gs_2 (A_g)$ are both positive, then we
can find a conformal metric $\hat g$ such that $\gs_2(A_{\hat g})$
is a positive constant; see also  \cite{GV03}.
For locally conformally flat closed manifolds,
 Li-Li \cite{LL03} and Guan-Wang  \cite{GW03a} proved that
  if $\gs_i(A_g) > 0$, for $1 \leq i \leq k,$ then we can
  find a conformal metric $\hat g$ such that $\gs_k(A_{\hat g})$ is
  constant. When $2k > n,$ the result was generalized by Gursky-Viaclovsky \cite{GV05}
   to non locally conformally flat closed manifolds;
   see also Trudinger-Wang \cite{TrWx05}.  Other  related works include
   Guan-Lin-Wang \cite{GLW04}, Ge-Wang \cite{GeW05} and Sheng-Trudinger-Wang \cite{STW07}.


 Let $M$ be a four-manifold  with boundary. The Gauss-Bonnet formula is 
 \beq \label{e:cgb}
  32 \pi^2 \chi (M, \partial M) = \int_M |\mathcal{W}|^2 + 16 (\int_M
  \gs_2(A_g) + \frac{1}{2} \oint_{\de M} \mathcal{B}_g),
  \eeq where $\mathcal{B}_g
   = \frac{1}{2} Rh - R_{nn} h - R_{\gc \ga \gc \gb} L^{\ga \gb} +
 \frac{1}{3} h^3 - h |L|^2+ \frac{2}{3} tr L^3,$
 and  $\int_M \gs_2 + \frac{1}{2} \oint_{\de M}
 \mathcal{B}_g$ is a conformal invariant. 
 We have the following existence result.
 Recall that the boundary $\de M$ is called umbilic if  $L_{\ga \gb} = \mu(x)
g_{\ga \gb},$ which is a conformal invariant condition.

 \begin{theorem}  \label{t:main}
 Let $(M, g)$ be a compact connected four-manifold with umbilic
 boundary. If $Y(M, \de M, [g])$ and
 $\int_M \gs_2 (A_g) + \frac{1}{2} \oint_{\de M}  \mathcal{B}_g$ are both positive,
 then there exists a metric $\hat{g} \in [g]$ such
 that  $\gs_2(A_{\hat g})$  is a positive constant and
  $\mathcal{B}_{\hat g}$ is zero.
 \end{theorem}

 We will prove a more general result than
 Theorem~\ref{t:main}.
 \begin{theorem} \label{t:main'}
 Let $(M,g)$ be a  compact connected four-manifold with umbilic
 boundary. Suppose that $(M, g)$ is not
  conformally equivalent to $(\mathbb{S}^+_4, g_c),$ where $g_c$ is the
  standard metric on the hemisphere. If  $Y(M, \de M, [g])$ and
 $\int_M \gs_2 + \frac{1}{2} \oint_{\de M} \mathcal{B}_g$ are both positive, then
 given a positive function $f$, there exists a metric $\hat{g} \in [g]$ such
 that  $\gs_2(A_{\hat g}) = f $  and
  $\mathcal{B}_{\hat g}$ is zero.
\end{theorem}

 An application of above theorem  to  Einstein
 manifolds is given in  Section \ref{subs:ccem}.

  For general $k$, we define  suitable boundary curvatures
 and  show variational properties of $\gs_k.$
      Let  $A^T = [A_{\ga \gb}]$
  be the tangential part of the Schouten tensor.
  Define
  \beq \label{e:b2}
    \mathcal{B}^2 = \left\{  \begin{array}{ll}
    \frac{2}{n-2} \gs_{2,1} (A^T, L)+
  \frac{2}{(n-2)(n-3)}\gs_{3,0} (A^T, L) &\, n \geq 4\\
   2 \gs_{2,1} (A^T, L)+ \frac{1}{3} h^3 - \frac{1}{2}h |L|^2
    & \, n=3,
  \end{array}\right.
  \eeq
  where $\gs_{i, j}$'s are the mixed symmetric functions;
  see Section~\ref{s:backgd}.
 For $k \geq 3,$ we define
 \beq \label{e:bk}
 \begin{array}{ll}
  \mathcal{B}^k = \sum_{i=0}^{k-1} C_1(n, k, i) \gs_{2k-i-1, i} (A^T, L)
  & \, n\geq 2k,
  \end{array}
  \eeq
   where $C_1 (n, k, i) = \frac{(2k-i-1)!(n-2k+i)!}{(n-k)!(2k-2i-1)!!\,i!}$
     and $!!$ stands for the double factorial.  When the boundary is
  umbilic, we define
 \beq \label{e:bk-u}
 \mathcal{B}^k = \sum_{i=0}^{k-1} C_2(n, k, i) \gs_i (A^{T})
\mu^{2k-2i-1} \eeq
 for all $n,$ where $C_2(n,k, i) = \frac{(n-i-1)!}{(n-k)! (2k-2i-1)!!}.$
  In Section~\ref{s:backgd}, we will show that the above two definitions of
  $\mathcal{B}^k$ coincide when the boundary is umbilic.

  Let
  $ \mathcal{F}_k (g) = \int_M \gs_k (A) + \oint_{\de M} \mathcal{B}^k$
   and $\mathcal{M} = \{g: g\in [g_0], V_g = 1\}.$
 \begin{theorem}\label{t:variation}
  Let $(M, g_0)$ be a compact manifold of dimension $n \geq 3$ with boundary.\\
 (a) Suppose $n \neq 4.$ Then $g$ is a critical point of $\mathcal{F}_2 \mid_\mathcal{M}$
     if and only if $g$  satisfies\\
 \hspace*{15pt}  $\gs_2 (A_g)= \text{constant}$  in  $M$
 with $  \mathcal{B}^2_g = 0$ on $\partial M.$\\ 
  (b) Suppose $n > 2k$ and $M$ is a locally conformally flat compact manifold.
   Then $g$ is a \\
   \hspace*{15pt} critical point of $\mathcal{F}_k \mid_\mathcal{M}$
     if and only if $g$  satisfies
    $\gs_k (A_g)= \text{constant}$  in  $M$
 with $  \mathcal{B}^k_g = 0$\\
 \hspace*{15pt} on $\partial M.$\\
   (c) The statement of (b) is true for all $n \neq 2k$ if we
   assume in addition that the boundary\\
   \hspace*{15pt} is umbilic.
 \end{theorem}

  If we add local conformal invariants to $\mathcal{B}^k,$ similarly  we have:
 \begin{corollary}\label{c:variation}
  Suppose $\mathcal{L}$ is a curvature tensor on $\de M$
  satisfying $\mathcal{L}(\hat g) = e^{(2k-1)u}
  \mathcal{L}(g).$ Then under the same conditions as in Theorem~\ref{t:variation},
  $g$ is a critical point of $(\mathcal{F}_k+ \oint
  \mathcal{L})\mid_\mathcal{M}$  if and only if $g$ satisfies
  $ \gs_k (A_g)= \text{constant}$ in $M$  with
  $\mathcal{B}^k_g + \mathcal{L}= 0$ on $\partial M.$
 \end{corollary}

  For closed manifolds, Theorem~\ref{t:variation} was proved by
  Viaclovsky \cite{Via00}. He also showed that $\mathcal{F}_{\frac{n}{2}}$ is a conformal
 invariant associated to the Gauss-Bonnet formula.  We will show a generalization
of this fact for manifolds with boundary in Section~\ref{s:conf-inv} (Proposition~\ref{p:gb}).


 We  study the problem of 
 finding a conformal metric $\hat g$ such that $\gs_k (A_{\hat g})$
 is constant and $\mathcal{B}^k_{\hat g}= 0.$
 For $k=2$ and $n=4,$ Theorem~\ref{t:main} shows that
 the problem is solvable under some conformal invariant conditions because
 when the boundary is umbilic,  $\mathcal{B}= 2 \mathcal{B}^2.$ 
 We remark that the boundary condition we find here in general involves second derivatives, which is highly nonlinear. Such 
 boundary condition is rare in the literature.

 We introduce some definitions and then   state the result for general $k.$
 Let $W$ be a matrix with eigenvalues $\gl_1,\cdots,\gl_n.$
 For $k \leq n,$   $\gs_k (W) =$\; $ \sum_{i_1<\cdots< i_k} \gl_{i_1}\gl_{i_2}\cdots\gl_{i_k}$
  is called the
kth elementary symmetric function of the eigenvalues of $W$. 
 The set  $\Gamma^+_k = \{ \gl \,: \gs _i (\gl)> 0 , 1 \leq i \leq k \}$
 is  called the positive
$k$-cone, which is  is an open convex cone with vertex at the origin  \cite{Gar59}. 
Now  we can define   higher order Yamabe constants
 for  manifolds with boundary.  When $\{\hat g: \hat g \in [g], A_{\hat{g}} \in \Gamma_{k-1}^+\}$ is nonempty,
 let $\mathcal{Y}_k = \inf  \mathcal{F}_k (\hat g)$
     where $\inf$ is taken over metrics $g \in [g]$ with $A_{\hat{g}} \in \Gamma_{k-1}^+$ and $V_{\hat g} = 1.$
   When $\{\hat g: \hat g \in [g], A_{\hat{g}} \in \Gamma_{k-1}^+\}= \emptyset,$ let $\mathcal{Y}_k = - \infty.$
   We  denote $\mathcal{Y}_1 = Y(M, \de M, [g]).$ 
   For closed manifolds,  $\mathcal{Y}_k$ was defined by Guan-Lin-Wang \cite{GLW04}. 
   For locally conformally flat
   closed manifolds,  Guan-Lin-Wang \cite{GLW04} proved that if $\mathcal{Y}_k >
   0$ and $2k \leq n,$ there exists $\hat g \in [g]$ such that $\gs_k (A_{\hat g})=1.$
   For manifolds with boundary, we have:
 \begin{theorem}\label{t:lcf-inv} Let $(M, g)$ be a locally conformally flat compact manifold of dimension
 $n \geq 3$ with umbilic boundary. Suppose that $2k \leq n$ and
 $\mathcal{Y}_1, \cdots, \mathcal{Y}_k > 0.$ Then there exists a metric
 $\hat g \in [g]$ such that $\gs_k (A_{\hat g})= 1$ and $\mathcal{B}^k_{\hat g} =0$.
 \end{theorem}


 Proofs of Theorem~\ref{t:main}, \ref{t:main'} and \ref{t:lcf-inv} turn out  by solving some boundary value problems
  for fully nonlinear equations. Under the conformal change of
the metric $\hat g = e^{-2u}g$, the Schouten tensor $\hat A$ satisfies
 \beq \label{e:schten}
 \hat A= \hess u + du \otimes du - \frac{1}{2}|\gra u|^2 g +
 A_g.
 \eeq
   The second fundamental form satisfies
 $\hat L e^{u}= \frac{\de u}{\de n} g + L_g,$
 where $n$ is the
 unit inner normal. 
  When the boundary is umbilic,
 the  formula becomes
  $ \hat \mu e^{-u} = \frac{\de u}{\de n} + \mu_g.$
  We will show in Section~\ref{s:backgd} that when $A_g \in
 \Gamma^+_k$ and when the boundary is umbilic, then $\mathcal{B}^k_g = 0$ if and only if
 $h_g = 0.$ 
 Thus, the  problem becomes solving
 \begin {equation} \label {e:sigma}\left\{  \begin{array}{ll}
 \sigma_k^{\frac{1}{k}} (\hess u + du \otimes du - \frac{1}{2}|\gra u|^2 g +
 A_g)= e^{-2u} & in \, M  \\
 \frac{\de u}{\de n} + \mu_g = 0& on \,\partial M.
  \end{array}\right .
  \end {equation}

 We will prove boundary estimates for equations more general than
 (\ref{e:sigma}). We use Fermi 
 coordinates in a boundary neighborhood.
  Define the half ball  by $\overline{B}_r^+ =
  \{ x_n \geq 0, \sum_i x_i^2 \leq r^2\}$ and the segment on the
  boundary by $\Sigma_r =  \{ x_n = 0, \sum_i x_i^2 \leq r^2\}.$
  Let $f(x,z) : M^n \times \mathbb R \rightarrow \mathbb R^+.$
 Consider the equation
\begin {equation} \label {e:symm}\left\{  \begin{array}{ll}
 F(\hess u + du \otimes du - \frac{1}{2}|\gra u|^2 g +
 S(x))=  f(x, u) & in \, \overline {B}^+_r  \\
 \frac{\de u}{\de n} + \mu_g = \hat \mu\, e^{-u}& on \,\Sigma_r,
  \end{array}\right .
  \end {equation}
 where $F$ satisfies some structure conditions
 as we describe now.
 Let $\Gamma$ be an open convex cone in $\mathbb{R}^n$ with vertex at the origin
 satisfying $\Gamma^+_n \subset \Gamma \subset \Gamma^+_1.$
 Suppose that $F(\gl)= F(\gs_1(\gl),\cdots,\gs_n(\gl)) \in C^{\infty}(\Gamma) \cap C^0(\overline{\Gamma})$
 is a homogeneous symmetric
 function of degree one normalized with $F(e) = F(1,\cdots,1) = 1.$
 Assume that $F = 0$ on $\de \Gamma$ and $F$ satisfies the following in $\Gamma:$

 (S0) $F$ is positive;\par
 (S1) $F$ is concave (i.e., $\frac{\de^2 F}{\de \gl_i \de
 \gl_j}$ is negative semi-definite);\par
 (S2) $F$ is monotone (i.e., $\frac{\de F}{\de \gl_i}$ is
 positive);\par
 (S3) $\frac{\de F}{\de \gl_i}\geq \epsilon \frac{F}{ \gs_1},$ for some constant
     $\epsilon > 0,$ for all $i.$ \par
 In some case, we need an additional condition: \\ 
 \hspace*{15pt} (A) $\sum_{j \neq i} \frac{\de F}{\de \gl_j} \leq \rho \frac{\de F}{\de \gl_i},$
      for some $\rho > 0,$ for all $\gl \in \Gamma$ with $\gl_i \leq
      0.$\\
 It was shown in \cite{Chen05a} that $\binom{n}{k}^{-\frac{1}{k}}\, \gs_k ^{\frac{1}{k}}$
  satisfies the structure conditions (S0)-(S3) and (A) in
  $\Gamma^+_k$ with $\epsilon = \frac{1}{k}$ and  $\rho = (n-k).$

 We assume that $S(x)$ satisfies the following conditions on the boundary:

 (T0) $S_{\ga n} = \mu_{\ga};$ \par
 (T1) $S_{\ga \gb} + S_{nn} g_{\ga \gb} \leq  R_{\ga n \gb n};$
 \par
 (T2) $S_{\ga \gb, n} - 2 \mu S_{\ga \gb} \leq
  \mu_{\tilde{\ga} \tilde{\gb}}- R_{\ga n \gb n} \mu,$\\
  where $\mu_{\tilde{\ga} \tilde{\gb}}$ means covariant
  derivatives of $\mu$ with respect to the induced metric $g_{\ga \gb}$
  on the boundary.

  Denote

 $c_{inf}(r) = \inf_{x \in \overline {B}^+_r} f(x,u);$ \par
 $c_{sup}(r) = \sup_{x \in \overline {B}^+_r} ( f+ | \gra_x f(x, u)|+ |f_z (x, u)|+
    | \hess_x f(x, u)| + |\gra_x f_z (x, u)| +
 |f_{z z} (x, u)|).$

 \begin{theorem} \label{t:bdy}
 Let $F$ satisfy (S0)-(S3) in a corresponding cone $\Gamma$
 and $S(x)$ satisfy (T0)-(T2) on $\Sigma_r$.
 Suppose that $|\gra_x f| \leq \Lambda f$ and $|f_z| \leq \Lambda
 f$ for some number $\Lambda,$ and
 $\Sigma_r$ is umbilic with principal curvatures $\mu.$
 Suppose $u \in C^4$ is a solution to the equation (\ref{e:symm}).\\
 Case(a). If $\hat \mu = 0,$ then
 $$\sup_{x \in \overline{B}^+_{\frac{r}{2}}} \,( |\gra u|^2 + |\hess u| ) \leq
   C ,$$
  where $C= C(r, n, \epsilon, \mu, \Lambda,
   \|S\|_{C^2 (\overline {B}^+_r)}, \|g\|_{C^3}, c_{sup} (r))$ \\
 Case(b).  Suppose that $F$ satisfies
 the additional condition (A) and $\Gamma_2^+ \subset \Gamma.$
  If $\hat \mu$ is a positive constant, then
  $$\sup_{x \in \overline{B}^+_{\frac{r}{2}}} \,( |\gra u|^2 + |\hess u| ) \leq
   C, $$
  where $C = C(r, n, \epsilon, \rho, \mu, \hat \mu, \Lambda,
  \|S\|_{C^2 (\overline {B}^+_r)}, \|g\|_{C^3}, \inf_{\overline {B}^+_r} u,
  c_{sup} (r)).$
 \end{theorem}

 When the manifolds are locally conformally flat on the boundary, we will show
 in Section~\ref{s:backgd} that  $A_g$
 satisfies the conditions (T0)-(T2).
 Denote the Weyl tensor by $\mathcal{W}_{ijkl}$
 and the Cotten tensor by $\mathcal{C}_{ijk} = A_{ij,k} - A_{ik,j}.$
 Then we have  the following Corollary.
  \begin{corollary}\label{c:lcfbdy}
  Let $F$ satisfy (S0)-(S3) in a corresponding cone $\Gamma.$
 Suppose that $\Sigma_r$ is umbilic with principal curvatures $\mu$ and $n$ is the unit
 inner normal with respect to $g.$ Suppose $\mathcal{W}_{ijkl}= 0$
 and $\mathcal{C}_{ijk}= 0$ on $\Sigma_r.$
 Let $u \in C^4$ be a solution to the equation
 \begin {equation} \left\{  \begin{array}{ll}
 F(\hess u + du \otimes du - \frac{1}{2}|\gra u|^2 g + A_g
 )=  f(x) e^{-2u} & in \, \overline {B}^+_r  \\
 \frac{\de u}{\de n} + \mu =  \hat \mu\, e^{-u} & on \, \Sigma_r.
  \end{array}\right .
  \end {equation}
 Case(a). If $\hat \mu = 0,$ then
 $$\sup_{x \in \overline{B}^+_{\frac{r}{2}}} \,( |\gra u|^2 + |\hess u| ) \leq
   C,$$
  where $C$ depends on $ r, n, \epsilon, \mu, \inf_{\overline {B}^+_r} u,
  \|g\|_{C^4}, \| f\|_{C^2(\overline {B}^+_r)}$ and
  $\inf_{\overline {B}^+_r} f.$\\
 Case(b).  Suppose that $F$ satisfies
 the additional condition (A) and $\Gamma_2^+ \subset \Gamma.$
  If $\hat \mu$ is a positive constant, then
  $$\sup_{x \in \overline{B}^+_{\frac{r}{2}}} \,( |\gra u|^2 + |\hess u| ) \leq
   C, $$
  where $C$ depends on $ r, n, \epsilon, \rho, \mu, \hat \mu, \inf_{\overline {B}^+_r} u,
   \|g\|_{C^4}, \| f\|_{C^2(\overline {B}^+_r)}$ and $\inf_{\overline {B}^+_r} f.$
  \end{corollary}

 The next estimates concern the $\gs_2$ equation.
 Let $A^t = A + \frac{1-t}{2} (tr A) g;$ see \cite{GV03}. Under the
 conformal change,  the tensor $\hat A^t$
 satisfies
 $$\hat A^t = \hess u + d u \otimes du - \frac{1}{2} |\gra u|^2 g
 + \frac{1-t}{2} (\Delta u - \frac{n-2}{2}|\gra u|^2)g + A^t.$$
 Consider the equation
  \begin {equation} \label{e:bdy1} \left\{  \begin{array}{ll}
  \gs_2^{\frac{1}{2}}(\hess u + du \otimes du - \frac{1}{2}|\gra u|^2 g
 + \frac{1-t}{2} (\Delta u - \frac{n-2}{2}|\gra u|^2)g + A^t+ S)= f(x, u) & in \, \overline {B}^+_r \\
 \frac{\de u}{\de n} + \mu_g = 0 & on \, \Sigma_r,
  \end{array}\right .
  \end {equation} where $S(x)$ is a $(0,2)$-tensor and $f(x, u)$ is positive.

 \begin{theorem} \label{t:bdy1} Let $n \geq 4.$ Suppose that $\Sigma_r$ is umbilic with principal
  curvatures $\mu.$   Let $u^t \in C^4$ be a solution to the equation
  (\ref{e:bdy1}).\par

   (a) When $t = 1,$ we have
   $$ \sup_{x\in \overline{B}_{\frac{r}{2}}^+}  |\gra u|^2
   \leq C_3 ,$$
   where $C_3 = C_3( n, r, \|g\|_{C^4}, \|S\|_{C^2 (\overline {B}^+_r)}, c_{\sup}(r))$
   but is independent of $c_{\inf}(r).$\par

   (b) Let $-\Theta \leq t \leq 1.$  Suppose in addition that $S$ satisfies $S_{\ga n}=0$ and
  $g^{\ga \gb} (S_{\ga \gb, n} - 2 \mu S_{\ga \gb}) \leq 0$ on
  $\Sigma_r.$ Then
   $$\sup_{x\in \overline{B}_{\frac{r}{2}}^+} (|\gra u|^2+ |\hess u|) \leq C_4,$$
    where $C_4 = C_4(n, r, \Theta, \|g\|_{C^4}, \|S\|_{C^2 (\overline {B}^+_r)},
    c_{\sup}(r), c_{\inf}(r)).$
  \end{theorem}


 The main technique we use in proving Theorem~\ref{t:bdy} and \ref{t:bdy1}
 is to derive boundary $C^2$ estimates directly from boundary $C^0$ estimates.
 Such idea has appeared before in the work by Chen \cite{Chen05} for local $C^2$
estimates for a large class of equations. (See \cite{GW03a} for a related work.) The same idea has also
been applied to boundary estimates in \cite{Chen05a}.
To control
boundary behaviors, we do not construct a barrier function.
Instead, we estimate the third derivatives uniformly on the boundary.  Then 
the maximum of second derivatives must happen in the interior.

 Finally, we remark that  the conformal invariants condition in Theorem~\ref{t:main}, \ref{t:main'} and \ref{t:lcf-inv}
  is necessary. A counterexample can be constructed on a cylinder if  the condition does not hold.
 We  also remark that the Dirichlet problem for  the Schouten tensor equations  was studied by
 Guan \cite{Gb05}.  
 The Neumann problems and  non-Dirichlet problems are, on the other hand, not yet well studied.

 This paper is organized as follows.
  We start with some background in Section~\ref{s:backgd}.
  In Sections~\ref{s:4mfd}, we prove  Theorems~\ref{t:main},
  \ref{t:main'} and their application. We give proofs of
  Theorems~\ref{t:variation} and Corollary~\ref{c:variation}
  in Section~\ref{s:variation}. In Section~\ref{s:conf-inv}, we
  prove  Theorem~\ref{t:lcf-inv} and Proposition~\ref{p:gb}.
  At the end, we prove boundary
 estimates. The proofs of Theorem~\ref{t:bdy} and
 Corollary~\ref{c:lcfbdy}, and Theorem~\ref{t:bdy1} are in
 Sections~\ref{s:bdy} and \ref{s:bdy1}, respectively.

 \textbf{Acknowledgments:} Part of the work in this paper is
 in the author's thesis at Princeton University.
 The author is grateful to her advisor, Alice
 Chang, for her support, help and patience.

\section{Background} \label{s:backgd}

  We give some basic facts about homogeneous symmetric functions.
\begin{lemma}(see \cite{Chen05}).\label{l:sym}
 Let $\Gamma$ be an open convex cone with vertex at the origin
 satisfying $\Gamma^+_n \subset \Gamma$ ,and let
 $e = (1, \cdots, 1)$ be the identity.
 Suppose that $F$ is a homogeneous symmetric function of degree one
 normalized with $F(e)= 1,$ and that $F$ is concave in $\Gamma.$
 Then\par

 (a) $\sum_i \gl_i \frac{\de F(\gl)}{\de \gl_i} = F(\gl), \quad$ for $\gl \in
 \Gamma;$\par

 (b) $\sum_i \frac{\de F(\gl)}{\de \gl_i} \geq F(e) = 1, \quad$
 for $\gl \in \Gamma.$
\end{lemma}
 Now we list further properties of elementary symmetric functions.
 \begin{lemma} \label{l:sigma}(see \cite{Chen05}).
  Let $G =\gs_k ^{\frac{1}{k}}, k \leq n.$ Then\par

  (a) $G$ is positive and concave in $\Gamma^+_k.$\par
  (b) $G$ is monotone in $\Gamma^+_k,$ i.e., the matrix
  $G^{ij} = \frac {\de G}{\de W_{ij}}$ is positive definite.\par
  (c) Suppose $\gl \in \Gamma^+_k.$ For $0 \leq l < k \leq n,$
    the following is the Newton-MacLaurin inequality
    $$ k(n-l+1) \gs_{l-1} \gs_k \leq l(n-k+1) \gs_l \gs_{k-1}.$$
\end{lemma}

 Let $W$ be an $m \times m$
 matrix.  $T_k(W)=\gs_k\,I - \gs_{k-1} W + \cdots + (-1)^k W^k$ is called
the $k$th Newton tensor of $W$; \cite{Reilly}. We have the recursive formula
$T_k(W) = \gs_k(W) \, I - T_{k-1}(W) W.$ Furthermore, $\frac{\de
\gs_k (W)}{\de W_{ij}} = T_{k-1}^{ij}(W)$ and $tr \; T_k(W) =
(m-k) \gs_k(W).$ 

  We introduce some more notations. Given an $n\times n$ matrix $A,$ denote the upper left $(n-1) \times (n-1)$
  sub-matrix by $A^T= [A_{\ga \gb}].$ The Greek letters
  $1\leq \ga,\gb,\gc \leq n-1$
 stand for the tangential indices and the
 letters $1 \leq i,j,k \leq n$ stand for the full indices
 unless otherwise noted.  The Kronecker symbol $\left(  \begin{array}{l}
    i_1 \cdots i_q\\
    j_1 \cdots j_q
    \end{array} \right) $ is defined as in   \cite{Reilly}.
\begin{lemma} \label{l:Tij} Let $A$ be an $n \times n$ matrix.\par
 (a) $\gs_q (A) = \frac{1}{q!} \sum
   \left(  \begin{array}{l}
    i_1 \cdots i_q\\
    j_1 \cdots j_q
    \end{array} \right) A^{j_1}_{i_1}\cdots A^{j_q}_{i_q};$\par
 (b) $ T_q (A)^i_j = \frac{1}{q!} \sum
   \left(  \begin{array}{l}
    i_1 \cdots i_q \,i\\
    j_1 \cdots j_q \,j
    \end{array} \right) A^{j_1}_{i_1}\cdots A^{j_q}_{i_q};$\par
 (c) $T_q (A)^n_n = \gs_q (A^T);$\par
 (d) $T_q (A)^{\ga}_n= - T_{q-1} (A^T)^{\ga}_{\gb} A^{\gb}_n.$
\end{lemma}
 \begin{proof}
 For (a) and (b), see \cite{Reilly}.
  (c) is directly from (a) and (b).
   (d) follows by an observation that
 $  T_q (A)_n^{\ga} 
    = \frac{1}{(q-1)!} \sum
   \left(  \begin{array}{l}
    i_1 \cdots i_{q-1} \, n\; \ga\\
    j_1 \cdots j_{q-1}\, j_q \,n
    \end{array} \right) A^{j_1}_{i_1}\cdots A^{j_{q-1}}_{i_{q-1}}
    A^{j_q}_n.\\
     $
 \end{proof}
 We define the mixed symmetric functions and Newton tensors:
 \begin{definition} \label{d:mix-sym}Let $A$ and $B$ be $m \times m$ matrices. Then\par
 $\gs_{q,r} (A, B) = \frac{1}{q!} \sum
   \left(  \begin{array}{l}
    i_1 \cdots i_q\\
    j_1 \cdots j_q
    \end{array} \right) A^{j_1}_{i_1}\cdots A^{j_r}_{i_r}
    B^{j_{r+1}}_{i_{r+1}}\cdots B^{j_q}_{i_q};$\par
 $ T_{q,r} (A, B)^i_j = \frac{1}{q!} \sum
   \left(  \begin{array}{l}
    i_1 \cdots i_q \,i\\
    j_1 \cdots j_q \,j
    \end{array} \right) A^{j_1}_{i_1}\cdots A^{j_r}_{i_r}
     B^{j_{r+1}}_{i_{r+1}}\cdots B^{j_q}_{i_q}.$
 \end{definition}
 Denote a variation of a tensor $A$ by $A'.$ The next lemma
  is used in proving Theorem~\ref{t:variation}.
\begin{lemma} \label{l:variation}
 Let $A$ and $B$ be $m \times m$ matrices. Suppose that ${A'}^j_i= k A^j_i \phi +
 M^j_i$ and ${B'}^j_i= l B^j_i \phi + N^j_i.$ Then\par
 (a) $T_{q,r} (A, B)^i_j A^j_i= (q+1) \gs_{q+1, r+1} (A, B)$\par
 $\begin{array} {ll}
  \text{(b)}\, \gs'_{q+1, r+1} (A, B) = &(k(r+1)+ l(q-r)) \gs_{q+1, r+1} (A, B)
 \phi\\
   & + \frac{r+1}{q+1} T_{q, r} (A, B)^i_j M^j_i+ \frac{q-r}{q+1} T_{q, r+1} (A, B)^i_j N^j_i
   \end{array}$\par
 (c) $\gs'_{q+1}(A)= k(q+1) \gs_{q+1}(A)\phi + T_q(A)^i_j M^j_i.$
\end{lemma}
 \begin{proof}
  (a) follows by definitions; see \cite{Reilly}, and (c) follows by
  (b) by letting $r=q$. 
  For (b), 
  $$\begin{array} {l}
  T'_{q,r}(A, B)^j_i = (kr+l(q-r)) T_{q,r}(A,B)^i_j \phi
      + \frac{1}{q!} \sum \left(  \begin{array}{l}
    i_1 \cdots i_q \,i\\
    j_1 \cdots j_q \,j
    \end{array} \right)\times \\
    \left[ \sum_{k=1}^r A^{j_1}_{i_1}\cdots M^{j_k}_{i_k}\cdots A^{j_r}_{i_r}
     B^{j_{r+1}}_{i_{r+1}}\cdots B^{j_q}_{i_q} +
     \sum_{k= r+1}^q A^{j_1}_{i_1}\cdots A^{j_r}_{i_r}
     B^{j_{r+1}}_{i_{r+1}}\cdots N^{j_k}_{i_k}\cdots B^{j_q}_{i_q}
     \right].
  \end{array}$$
  Using (a) and the formula above, we then have
 \begin{eqnarray*}
 (q+1) \gs'_{q+1, r+1} (A, B)   &=& (kr + l(q-r)) T_{q,r}(A, B)^i_j A^j_i \phi + rT_{q, r} (A,
  B)^i_j M^j_i\\
  &+& (q-r) T_{q, r+1} (A, B)^i_j N^j_i + T_{q,r}(A, B)^i_j (k A^j_i\phi +
  M^j_i).
 \end{eqnarray*}
 Using (a) again gives the result.
 \end{proof}

Now we check that two definitions of $\mathcal{B}^k$'s,
(\ref{e:bk}) and (\ref{e:bk-u}),
  coincide when the boundary is umbilic.
 By definition,
 $$
  \gs_{q, r} (A^T, \mu g) 
    = \frac{(n-1-r)!}{q!(n-1-q)!} \sum_{i_1, \cdots, j_1 \cdots < n}
   \left(  \begin{array}{l}
    i_1 \cdots i_r\\
    j_1 \cdots j_r
    \end{array} \right) A^{j_1}_{i_1}\cdots A^{j_r}_{i_r}
    \mu^{q-r}.
 $$
   Therefore, $\gs_{q, r} (A^T, \mu g) = \frac{r! (n-1-r)!}{q!(n-1-q)!} \gs_r(A^T) \mu^{q-r}$
   and  $\sum_{i=0}^{k-1} C_1(n,k, i) \gs_{2k-i-1, i} (A^T, \mu g)=
   \sum_{i=0}^{k-1} \frac{(n-1-i)!}{(n-k)!(2k-2i-1)!!} \gs_i(A^T) \mu^{q-r}$ $=
   \sum_{i=0}^{k-1} C_2(n,k,i) \gs_i(A^T) \mu^{q-r}.$


Next, we show some properties of curvatures on the boundary.
 We review two of the fundamental equations:
$R_{ijkl,m} + R_{ijmk,l} + R_{ijlm,k} = 0$  (Bianchi  identity) and
 $R_{\ga \gb \gc n} =   L_{\ga \gc, \gb} -  L_{\gb \gc,
  \ga}$ (Codazzi  equation),
 where $n$ is the unit inner normal with respect to $g.$
 In Fermi (geodesic) coordinates, the metric is expressed as $g =
dx^n dx^n + g_{\ga \gb} dx^{\ga} dx^{\gb}.$  The Christoffel
symbols satisfy
 \beq \label{e:christ}
 \Gamma_{\ga \gb}^n = L_{\ga \gb}, \quad
 \Gamma_{\ga n}^{\gb} = -L_{\ga \gc} g^{\gc \gb}, \quad
   \Gamma_{\ga n}^n = 0
  \eeq on the boundary.
 When the boundary is umbilic, they become
 \beq \label{e:christ-u}
 \Gamma_{\ga \gb}^n = \mu g_{\ga \gb}, \quad
 \Gamma_{\ga n}^{\gb} = -\mu \gd_{\ga \gb}, \quad
   \Gamma_{\ga n}^n = 0.
  \eeq
 We denote the tensors and covariant differentiations
 with respect to the induced metric $g_{\ga \gb}$ on the boundary by
 a \emph{tilde} (e.g. $\tilde{R_{\ga \gb}},$ $\mu_{\tilde{\ga} \tilde{\gb}}$).
 Then the  Christoffel symbols  satisfy
 \beq \label{e:christ-i}
 \widetilde{\Gamma}_{\ga \gb}^{\gc} = \frac{1}{2} g^{\gc \gd} (\frac{\de g_{\ga \gd}}{\de x_{\gb}}
  + \frac{\de g_{\gb \gd}}{\de x_{\ga}} - \frac{\de g_{\ga \gb}}{\de
  x_{\gd}})=
  \Gamma_{\ga \gb}^{\gc}.
 \eeq  We also denote the Laplacian in the
 induced metric by $\widetilde{\Delta}.$

 The next lemma gives us the relation between $\mathcal{B}^k_g$ and $h_g.$
\begin{lemma} \label{l:b=0}
 Let $(M, g)$ be a  compact manifold  with umbilic
 boundary. If $h_g = 0$ on the boundary, then we have $\mathcal{B}^k_g=
 0$. Conversely, if $\mathcal{B}^k_g= 0$ on the boundary and if in addition
 $A_g \in \Gamma_k^+$, then $h_g= 0.$
 \end{lemma}
\bpf
  Let $L_{\ga \gb}= \mu g_{\ga
  \gb}.$ By Definition~\ref{d:mix-sym},
  $
  \gs_{2,1} (A^T, L) = \gs_{2,1} (A^T, \mu g) = \frac{1}{2} tr
  T_1 (A^T) \mu= \frac{n-2}{2} \gs_1 (A^T) \mu.
  $ Therefore, when $n \geq 4,$ we obtain
  $\mathcal{B}^2 = \gs_1 (A^T) \mu + \frac{2}{(n-2)(n-3)} \gs_3 (\mu g)
  = (\gs_1 (A^T) + \frac{n-1}{3} \mu^2)\mu.$ When $n=3$,
  $\mathcal{B}^2 = \gs_1 (A^T) \mu + \frac{1}{3}(2 \mu)^3 -\frac{1}{2} (2 \mu) (2\mu^2)
  =  (\gs_1 (A^T) + \frac{2}{3}\mu^2)\mu.$
 For $k \geq 3,$ we have $\mathcal{B}^k = (\sum_{i=0}^{k-1} C_2(n, k, i) \gs_i (A^{T})
 \mu^{2k-2i-2})\mu,$ where $C_2(n,k, i)$ is positive.

 Since $(n-1)\mu = h,$ if $h=0,$ then clearly $\mathcal{B}^k_g = 0.$ When
  $A_g \in \Gamma_k^+,$ by Lemmas~\ref{l:sigma} and \ref{l:Tij}
  we have $T_i(A)^n_n= \gs_i (A^T)$ is positive for $i < k$.
  As a result, $\mathcal{B}^k_g=0$ implies $h= 0.$
\epf

 We verify that the Schouten tensor satisfies conditions (T0)-(T2)
 when $\mathcal{W}=0$ and $\mathcal{C}=0$ on the boundary.

 \begin{lemma} \label{l:bdy} Suppose that the boundary is umbilic. Let
 $n$ be the unit inner normal with respect to $g.$
 Then\par
 (a) $A_{\ga n} = \mu_{\ga}$ on $\de M;$\par
 (b) $\mu_{\tilde{\ga} \tilde{\gb}} = A_{\ga n, \gb}+  A_{nn} \mu g_{\ga \gb} -  A_{\ga \gb}\mu$
    on $\de M;$ \par
 (c) If $\mathcal{W}=0$  on $\de M,$ then
    we have $R_{n \ga n \gb}= A_{\ga \gb}+ A_{nn} g_{\ga \gb}$ on
    the boundary. If in addition $\mathcal{C}= 0$ on $\de M,$ then
    $A_{\ga \gb, n}- 2\mu A_{\ga \gb}= \mu_{\tilde{\ga} \tilde{\gb}}
    -R_{\ga n \gb n}\mu.$
\end{lemma}
 \begin{proof} By the Codazzi equation,
  we get $R_{\ga n}= (n-2) \mu_{\ga}$ and $A_{\ga n}= \mu_{\ga}.$

 For (b), we  use (a), (\ref{e:christ-u}) and (\ref{e:christ-i}) to get
  $$
   \mu_{\tilde{\ga} \tilde{\gb}} = \de_{\gb} A_{\ga n} -
   \Gamma_{\ga \gb}^{\gc} \mu_{\gc}= (A_{\ga n, \gb} + \Gamma_{\ga
   \gb}^l A_{l n} + \Gamma_{\gb n}^l A_{\ga l}) - \Gamma_{\ga
   \gb}^{\gc} \mu_{\gc}
   = A_{\ga n, \gb}+ A_{nn} \mu g_{\ga \gb}- A_{\ga
   \gb}\mu.
  $$

  For (c),  using the curvature decomposition formula
  $R_{ijkl} = \mathcal{W}_{ijkl} + A_{ik} g_{jl} + A_{jl} g_{ik}
 - A_{il} g_{jk} - A_{jk} g_{il},$ we first get
  $ R_{n \ga n \gb}=  A_{nn} g_{\ga \gb}+ A_{\ga \gb}$
  when $\mathcal{W}= 0.$ If in addition $\mathcal{C}= 0,$ then
  $A_{\ga \gb, n}- 2\mu A_{\ga \gb}= A_{\ga n, \gb} - 2 \mu A_{\ga \gb}
  =A_{\ga n,\gb} + \mu A_{nn} g_{\ga \gb}
    - A_{\ga \gb}\mu -R_{\ga n \gb n}\mu. $
 \end{proof}
The next lemma will be used in proving Theorem~\ref{t:bdy} and
\ref{t:bdy1}.
\begin{lemma}
 Suppose $\de M$ is umbilic. Let $u$ satisfy $u_n = - \mu + \hat \mu
 e^{-u},$ where $\hat \mu$ is constant. Then we have
  \beq \label{e:nga}
 u_{n \ga} = - \mu_{\ga}+ \mu u_{\ga} - \hat \mu u_{\ga}
 e^{-u};
 \eeq
 \beq \label{e:ngagb} \begin{array}{ll}
 u_{\ga \gb n} &= (2 \mu -\hat \mu e^{-u}) u_{\ga \gb} - \mu u_{nn} g_{\ga \gb}
     + \hat \mu u_{\ga} u_{\gb} e^{-u}
     -\mu_{\tilde{\ga} \tilde{\gb}}
      + \mu_{\ga} u_{\gb} + \mu_{\gb} u_{\ga} \\
      & - \mu_{\gc} u_{\gc} g_{\ga \gb}
     + R_{n \gb \ga n} (- \mu + \hat \mu e^{-u}) - \mu (-\mu + \hat \mu e^{-u})^2 g_{\ga \gb}.
 \end{array}
 \eeq
 \end{lemma}
 \bpf
  By (\ref{e:christ-u}),
  $u_{n \ga} = \de_{\ga} u_n - \Gamma^j_{\ga n} u_j = \de_{\ga} u_n + \mu u_{\ga}
   = - \mu_{\ga} - \hat \mu u_{\ga} e^{-u} + \mu u_{\ga}.$
  For (\ref{e:ngagb}), by (\ref{e:christ-u}) and (\ref{e:christ-i})
  $
   u_{n \ga \gb} = \de_{\gb} u_{n \ga} - \Gamma^j_{\gb n} u_{j
   \ga}- \Gamma^j_{\ga \gb} u_{n j} = \de_{\gb} u_{n \ga} + \mu
   u_{\gb \ga}- \tilde{\Gamma}^{\gc}_{\ga \gb} u_{n \gc} -\mu u_{nn} g_{\ga
   \gb}.
  $ Now by  (\ref{e:nga}), (\ref{e:christ-u}) and $u_n = - \mu + \hat \mu
  e^{-u},$
   \begin{align}
   u_{n \ga \gb} &=  u_{n \tilde{\ga} \tilde{\gb}} + \mu
   u_{\gb \ga} -\mu u_{nn} g_{\ga
   \gb}\notag\\
     &=  - \mu_{\tilde{\ga} \tilde{\gb}} - \hat \mu u_{\ga \gb}
   e^{-u} + \hat\mu u_{\ga} u_{\gb} e^{-u}+ \mu_{\gb} u_{\ga} + 2\mu u_{\ga \gb}
       - \mu u_{nn} g_{\ga\gb} - \mu (- \mu + \hat \mu e^{-u})^2 g_{\ga \gb}
      \notag.
   \end{align}
   On the other hand,  using the Codazzi equation gives
   $u_{\ga \gb n} = u_{n \ga \gb} + R_{n \gb \ga j} u_j  
      =  u_{n \ga \gb} + \mu_{\ga} u_{\gb} - \mu_{\gc} u_{\gc} g_{\ga \gb} + R_{n \gb
\ga n} (-\mu + \hat \mu e^{-u}).
   $
  Combing above formulas yields (\ref{e:ngagb}).
  \epf
 The last lemma of this section is a boundary version of the
 Bianchi identity.
\begin{lemma} \label{l:bdy1}
 Suppose that the boundary $\de M$ is umbilic and under a
conformal change $\hat g= e^{-2u} g$, $\hat L_{\ga \gb} = 0$ near
a boundary point $x_0.$ Then
 $ g^{\ga \gb}\hat A _{\ga \gb ,n} = 2 \mu g^{\ga \gb}
 \hat A_{\ga \gb}$ at $x_0.$
\end{lemma}

\begin{proof}

 We denote the covariant differentiation with respect to the new
 metric $\hat g$ by $\hat \gra.$ Since $\hat L_{\ga \gb}= 0,$ by the Codazzi equation
 $ \hat R_{\ga \gb \gc n} =0.$
 Therefore, we have
  $\hat R_{\ga n} = 0$ and $\hat A_{\ga n}= 0$ at $x_0.$
 Hence,
 $\hat \gra_{\gb} \hat R_{\ga n} = \de_{\gb} \hat R_{\ga n} - \hat
\Gamma_{\gb \ga}^k \hat R_{k n} - \hat \Gamma_{\gb n}^k \hat
R_{\ga k} = - \hat \Gamma_{\gb \ga}^n \hat R_{n n} - \hat
\Gamma_{\gb n}^{\gc} \hat R_{\ga \gc}.$
 By (\ref{e:christ-u}), both $\hat \Gamma_{\gb \ga}^n$ and $\hat \Gamma_{\gb n}^{\gc}$
are zero. Thus, we have $\hat \gra_{\gb} \hat R_{\ga n}= 0.$

On the other hand, by the Bianchi identity,
  $0 = \hat \gra_n \hat R_{i \ga k \gb} + \hat \gra_k \hat R_{i \ga \gb n} +
 \hat \gra_{\gb} \hat R_{i\ga n k}.$ Contracting indices $i$ and
 $k$ gives
   $ 0= \hat \gra_n \hat R_{\ga \gb} + \hat g^{ik} \hat \gra_k \hat R_{i \ga \gb
 n} - \hat \gra_{\gb} \hat R_{\ga n}.$
 Noting that $\hat \gra_{\gb} \hat R_{\ga n}= 0$ and $\hat g^{\ga n}
 =0$, contract indices $\ga$ and $\gb$ to get
 $0 = \hat g^{\ga \gb} \hat \gra_n \hat R_{\ga \gb} - \hat g^{ik} \hat \gra_k \hat R_{in}
   = \hat g^{\ga \gb} \hat \gra_n \hat R_{\ga \gb} - \hat g^{nn} \hat \gra_n \hat R_{nn}.$
 Therefore,
 \beq \label{e:A_gagbn}
 \hat g^{\ga \gb} \hat \gra_n \hat A_{\ga \gb} = \frac{1}{(n-2)}
(\hat g^{\ga \gb}
 \hat \gra_n \hat R_{\ga \gb}- \frac{1}{2} \hat R_n) =
 \frac{1}{2 (n-2)} (\hat g^{\ga \gb}
 \hat \gra_n \hat R_{\ga \gb}- \hat g^{nn} \hat \gra_n \hat R_{nn})= 0.
 \eeq
 Using $\hat A_{\ga n}= 0$, $\hat \Gamma_{\ga \gb}^n= \hat \Gamma_{\gb n}^{\ga}=0$
 and  (\ref{e:christ-u}), we finally arrive at
 $$
 0 = \hat g^{\ga \gb} \hat \gra_n \hat A_{\ga \gb} = 
    \hat g^{\ga \gb} \de_n \hat A_{\ga \gb}
    = \hat g^{\ga \gb} (\hat A_{\ga \gb , n}+  \Gamma_{n \ga}^k \hat A_{k \gb}
    + \Gamma_{n \gb}^k \hat A_{k \ga}) 
    = \hat g^{\ga \gb} (\hat A_{\ga \gb ,n}
   - 2 \mu \hat A_{\ga \gb}).
 $$
\end{proof}

\section{Four-manifolds} \label{s:4mfd}

 In this section, we only consider $n=4.$ We prove Theorem~\ref{t:main'} and Corollary~\ref{c:ccem}.
 The proof of Theorem~\ref{t:main'} consists of two propositions:
 \begin{proposition}\label{p:pos}
 Let $( M, g)$ be  a  compact connected four-manifold with umbilic
 boundary. If  $Y(M, \de M, [g])$ and
 $\int_M \gs_2 + \frac{1}{2} \oint_{\de M} \mathcal{B}_g$ are both
 positive, then there exists a metric $\hat g \in [g]$ such that
 $ R_{\hat g} > 0$, $\gs_2(A_{\hat g}) > 0,$ and the boundary is totally geodesic.
 \end{proposition}

 \begin{proposition} \label{p:f}
 Suppose $(M, g)$ is a compact connected four-manifold with
 totally geodesic boundary. If $R_g> 0$, $\gs_2(A_g)> 0,$ and
 $(M, g)$ is not conformally equivalent to $(\mathbb{S}^+_4, g_c),$
 then given a positive function $f$ there exists a metric
 $\hat g \in [g]$ such that  $\gs_2(A_{\hat g}) = f $  and
  $\mathcal{B}_{\hat g}$ is zero.
 \end{proposition}

 We will prove Propositions~\ref{p:pos} and
 \ref{p:f} in Subsections~\ref{subs:pos} and \ref{subs:f},
 respectively.

 \subsection{Conformal Metric Satisfying $\gs_2 > 0$ } \label{subs:pos}

  We will deform a Yamabe metric to the one satisfying the
  properties in Proposition~\ref{p:pos}. The deformation
  comes from a nice idea by Gursky-Viaclovsky \cite{GV03} for closed
  four-manifolds.
 \begin{proof}[Proof of Proposition~\ref{p:pos}]
 Let the background metric $g$ be a Yamabe metric. Thus, we have $R_g$ is
 a positive constant and the boundary is totally geodesic.

 Let $A^t =  A+ \frac{1-t}{2} (tr_g  A)g.$
    Let $\hat g = e^{-2u}g.$ For $n= 4,$ the tensor $\hat A^t$ satisfies
  $ \hat A^t
  = \hess u + du \otimes du  - \frac{1}{2} |\gra u|^2 g + \frac{1-t}{2}(\Delta u
  -|\gra u|^2 )g + A^t.
 $
  We can choose a large number $\Theta$
  such that
 $ A^{-\Theta} = \frac{1}{2} (Ric_g +\frac{\Theta}{6} R_g g) $ is positive
 definite. Let $f(x) = \gs_2^{\frac{1}{2}} (A^{-\Theta}_g)$.
 Thus,  $A^{-\Theta}_g \in \Gamma_2^+$ and $f$ is positive.
Consider the following path of equations for $- \Theta \leq t \leq
1:$
 \begin{equation} \label {e:pos}\left\{  \begin{array}{ll}
 \gs_2^{\frac{1}{2}} (\hess u + du \otimes du  - \frac{1}{2} |\gra u|^2 g + \frac{1-t}{2}(\Delta u
  -|\gra u|^2 )g + A^t_g)= f(x) e^{2u} & in \, M  \\
  \frac{\de u}{\de n} = 0 & on \,\de M.
  \end{array}\right .
  \end{equation}

 Let $\mathcal{S} = \{ t\in [-\Theta, 1]: \exists$ a solution $u \in C^{2, \ga}(M)$
 to (\ref{e:pos}) with $ \hat A^t \in \Gamma^+_2\}.$ At $t= -\Theta$, we have
 $u \equiv 0$ is a solution and $A^{-\Theta}_g \in \Gamma^+_2$. Therefore,
 $\mathcal{S}$ is nonempty.  Consider the
 linearized operator $\mathcal{P}^t: C^{2, \ga}(M) \cap \{\frac{\de u}{\de n} = 0$ on $\de M\}
 \rightarrow C^{\ga} (M).$ It was proved in \cite{GV03} (Proposition~2.2) that the
 linearized operator is elliptic with the strictly negative coefficient in the zeroth order term.
 By elliptic theory for Neumann 
 condition  \cite{GT}, the linearized operator is invertible.
 Hence,  $\mathcal{S}$ is open.
 If $\mathcal{S}$ is also closed, then we
 have a solution $u$ to (\ref{e:pos}) at $t= 1$ with $\hat A^1 =
 \hat A \in \Gamma^+_2.$ This gives  $\hat g = e^{-2u} g$
 satisfying  $\gs_2 (\hat A) > 0,$ $\hat R> 0$ and  $\hat \mu = 0.$  
 Thus, it remains to establish a priori estimates for solutions to
 (\ref{e:pos}) independent of $t.$

 (1) $C^0$ estimates.

   At the maximal point $x_0$ of $u,$ if $x_0$ is in the
   interior, we have $|\gra u| = 0$. If $x_0$ is at the boundary,
   since $\frac{\de u}{\de n} = 0$, we also have
   $|\gra u| = 0.$ Therefore,  we get that $\hess u (x_0)$ is negative
   semi-definite and $\Delta u(x_0) \leq 0.$
   By Lemma~\ref{l:sigma} (c),
   $$
    f(x_0) e^{2u(x_0)} = \gs_2^{\frac{1}{2}}(g^{-1} \hat A^t)
    \leq \frac{\sqrt{6}}{4} \gs_1 (g^{-1} \hat A^t) 
     = \frac{\sqrt{6}}{4} (3-2t) \Delta u + \frac{\sqrt{6}}{4} tr_g \,A^t_g
     \leq \frac{\sqrt{6}}{4} tr_g \,A^t_g \leq C,
   $$
   where in the second inequality we use $t \leq 1$ and $\Delta u \leq 0.$
  Hence, $u$ is  upper bounded.

   Now we prove the Harnack inequality.
   Let $H = |\gra u|^2$. If the maximum of $H$ is in
 the interior, then $\gra H =0,$ and $\hess H$ is negative
 semi-definite. If the maximum of $H$ is at the boundary, since
 $\frac{\de u}{\de n} =0$ and $\mu_g = 0,$ we have $u_{\ga n} = 0$
  and $H_n = 2 u_{\ga}u_{\ga n} + 2 u_n u_{nn} =0.$
 Thus, we also have that $\gra H =0,$ and $\hess H$ is negative
 semi-definite. Interior gradient estimates for (\ref{e:pos})
 were proved in \cite{GV03} (Proposition 4.1). We remark that
 the same proof works for boundary gradient estimates. The reason is
 that   at the maximal point once we  have $\gra H = 0,$
 and $\hess H$ is negative semi-definite, then the rest of
 computations in \cite{GV03} is the same regardless of the point being in
 the interior or on the boundary. Therefore, we get $|\gra u|< C$.
 Thus, $\sup_M u \leq \inf_M u + C.$

  To prove that $\sup_M u$ is
 lower bounded, integrating the equation gives
  $$
   C e^{4 \sup_M u}   \geq \int_M f^2 e^{4u}dV_g = \int_M \gs_2
   (g^{-1} \hat A^t) dV_g  = \int_M \gs_2 ( \hat g^{-1}\hat A^t) dV_{\hat
   g},$$
  where in the second equality we use $ dV_{\hat g} = e^{-4u}
  dV_g.$ Note that $\gs_2 ( \hat g^{-1}\hat A^t)= \gs_2(\hat A)+
  \frac{3}{2} (1-t)(2-t) \gs^2_1(\hat A).$ Thus, the above formula
  becomes
  $$
    C e^{4 \sup_M u} \geq \int_M ( \gs_2(\hat A)+
    \frac{3}{2} (1-t)(2-t) \gs^2_1( \hat A)) dV_{\hat g}
    \geq \int_M \gs_2 ( \hat A)dV_{\hat g}.
  $$
   Recall that  the conformal invariant
   $\int_M \gs_2 + \frac{1}{2} \oint_{\de M} \mathcal{B}_g$ is
   positive. Since $\hat \mu = 0,$  by Lemma~\ref{l:b=0} we get
   $\mathcal{\hat B} = 0.$  Finally, we have
   $$
    C e^{4 \sup_M u} \geq
    \int_M \gs_2 ( \hat A) dV_{\hat g} + \frac{1}{2} \oint_{\de M}
    \mathcal{\hat B} \,dS_{\hat g}
     = \int_M \gs_2 (A_g)dV_g + \frac{1}{2} \oint_{\de M}
     \mathcal{B}_g \,dS_g >0.
   $$

 (2) $C^2$ estimates.

  Interior $C^2$ estimates are proved in \cite{Chen05}.
  To get boundary $C^2$ estimates, we use Fermi coordinates in a
  tubular neighborhood $\de M \times [0, \iota]$ of the boundary.
  Note that $\de M$ is compact so $\iota$ is a positive
  number. Thus, by Theorem~\ref{t:bdy1} (b) (with $S = 0$) we obtain boundary $C^2$ estimates
  in each half ball $\overline {B}^+_r.$ Since $\de M$ is compact,
  there are finitely many local charts of the tubular neighborhood.
   We then get the required estimates.

 (3) $C^{\infty}$ estimates.

  Once we have $C^2$ bounds, the equation is uniformly elliptic and
 concave. Higher order regularity follows
 by standard elliptic theories; see \cite{Evan82},\cite{Kry83} and  \cite{LT86}.
 \end{proof}


 \subsection{Conformal Metric Satisfying $\gs_2 = f$ } \label{subs:f}

  In this subsection, we proof Proposition~\ref{p:f}.  We first prove a lemma.
 \begin{lemma} \label{l:conf}
  Let $(M, g)$ be a compact four-manifold with
  umbilic boundary.  Suppose $Y(M, \de M, [g]) > 0.$ Then
  $$\int_M \gs_2 (A_g) + \frac{1}{2} \oint_{\de M}  \mathcal{B}_g \leq 2 \pi^2.$$
  Moreover, the equality holds if and only if $(M, g)$ is
  conformally equivalent to $(\mathbb{S}^4_+, g_c),$ where $g_c$ is the
  standard metric on the hemisphere.
 \end{lemma}
 \begin{proof}
   Denote the volume of $(M, g)$ by $V_g.$
  Let $\tilde{g}$ be a Yamabe metric such that $R_{\tilde{g}}$ is constant and
  the boundary is totally geodesic.  It was proved by Escobar \cite{Es92} that
  \beq \label{i:yamb}
   Y(M,\de M, g) = \frac {\int_M R_{\tilde{g}}+ \oint_{\de M} 3 \mu_{\tilde{g}}}
   {V_{\tilde{g}}^{\frac{1}{2}}}=
   R_{\tilde{g}} V_{\tilde{g}}^{\frac{1}{2}} \leq Y(\mathbb{S}^4_+,\mathbb{S}^3, g_c) = 8 \sqrt 3 \pi.
   \eeq
  The equality holds if and only if $(M, g)$
   is conformally equivalent to $(\mathbb{S}^4_+, g_c).$
   Since $\mu_{\tilde{g}}= 0,$ by Lemma~\ref{l:b=0} we have $\mathcal{B}_{\tilde{g}}= 0.$
   Therefore,
   $
    \int_M \gs_2 (A_g) + \frac{1}{2}\oint_{\de M}  \mathcal{B}_g =
    \int_M \gs_2 (A_{\tilde{g}}) + \frac{1}{2}\oint_{\de M}  \mathcal{B}_{\tilde{g}}
    = \int_M \gs_2 (A_{\tilde{g}}).$
   Note that $\gs_2 (A) = \frac{1}{8} (\frac{1}{12} R^2- |E|^2),$
  where $E= Ric - \frac{1}{4} R g.$ 
   By (\ref{i:yamb}) we get
    $
    \int_M \gs_2 (A_g) + \frac{1}{2}\oint_{\de M}  \mathcal{B}_g =
     \frac{1}{8} \int_M (\frac{1}{12} R^2_{\tilde{g}}- |E_{\tilde{g}}|^2 )
    \leq \frac{1}{96} R^2_{\tilde{g}} V_{\tilde{g}} \leq 2 \pi^2.
    $
   The equality holds if and only if $(M, g)$
   is conformally equivalent to $(\mathbb{S}^4_+, g_c).$
 \end{proof}

 \bpf[Proof of Proposition~\ref{p:f}]
 Let $\hat A= \gra ^2 u + du \otimes du -\frac{1}{2}|\gra u|^2 g +
 A_g.$ Since $\mu_g= 0,$ the problem is equivalent to solve
 $$
 \left\{  \begin{array}{ll}
 \gs_2 ^{\frac{1}{2}} (\hess u + du \otimes du -\frac{1}{2}|\gra u|^2 g + A_g) =
 f(x)\, e^{-2u} & in \, M \\
 \frac{\de u}{\de n} = 0  & on \, \de M
 \end{array}\right .$$ with $\hat A \in \Gamma^+_2.$

   Denote the volume of $(M, g)$ by $V_g.$
 We will use a deformation motivated by
  \cite{GV04}, \cite{GV05} for closed manifolds.
 Let $S_g= (1-\zeta(t)) (\frac{1}{\sqrt 6} V^{\frac{2}{5}} g -A_g).$
  Consider the following path of equations for $0 \leq t \leq 1$ with $\hat A + S_g \in \Gamma^+_2$:
 \beq \label{e:defm}
  \left\{  \begin{array}{ll}
 \gs_2 ^{\frac{1}{2}} (\hess u + du \otimes du -\frac{1}{2}|\gra u|^2 g
  + A_g + S_g)
  =(1-t)(\int_M e^{- 5u})^{\frac{2}{5}} + \zeta(t) f \, e^{-2u} & in \, M \\
 \frac{\de u}{\de n} = 0  & on \, \de M,
 \end{array}\right .
 \eeq
 where  $\zeta(t) \in C^1 [0,1]$ satisfies $0 \leq \zeta \leq 1$,
 $\zeta(0)= 0,$ and $\zeta = 1$ for $t \geq \frac{1}{2}.$
  The Leray-Schauder
 degree is defined by considering the space $\{u \in C^{4, \ga} (M): \frac{\de
u}{\de n} = 0$ on $\de M \};$ see \cite{Chen05a}. 
We check that  at $t = 0$
the degree is nonzero.
 For closed manifolds, it was proved in \cite{GV04} that the
 degree is nonzero at $t=0.$ For manifolds with boundary,
 we remark that the same proof works. More specifically,
  at $t=0,$ (\ref{e:defm}) becomes
   $$
  \left\{  \begin{array}{ll}
 \gs_2 ^{\frac{1}{2}} (\hess u + du \otimes du -\frac{1}{2}|\gra u|^2 g
  + \frac{1}{\sqrt 6} V^{\frac{2}{5}} g)
  = (\int_M e^{- 5u})^{\frac{2}{5}}  & in \, M \\
 \frac{\de u}{\de n} = 0  & on \, \de M.
 \end{array}\right .
  $$
  By the boundary condition $\frac{\de u}{\de n}=0$
  if  the maximum (resp. minimum) of $u$ happens at the boundary, we still
  have $\gra u =0,$ and  $\hess u$ is negative (resp. positive)
  semi-definite. Hence, as in \cite{GV04} by the maximum principle,
  $u=0$ is the unique solution.

  Now the linearized operator $\mathcal{P}:
  C^{2, \ga}(M) \cap \{\frac{\de \phi}{\de n}=0$ on $\de M \} \rightarrow C^{\ga} (M)$
  at $u=0$ is
  $ \mathcal{P} (\phi) = \frac{\sqrt 6}{4}  \Delta \phi + 2 V_g^{-\frac{2}{5}}\int_M \phi.$
  Then the rest of the proof of
  showing the degree is nonzero at $t=0$ follows from \cite{GV04}.
  Consequently, the problem reduces to establishing a priori
 estimates for (\ref{e:defm}).

 Suppose we have uniform $C^0$ bounds for (\ref{e:defm}).
  By \cite{Chen05}, we get interior $C^2$ estimates.
 For boundary $C^2$ estimates, we check that $S$ satisfies the
 condition in Theorem~\ref{t:bdy1} (b). Since $\mu = 0,$ by
 Lemma~\ref{l:bdy} (a) we have $S_{\ga n}= 0$ and by (\ref{e:A_gagbn})
 $g^{\ga \gb} S_{\ga \gb, n}= - (1-\zeta(t)) g^{\ga \gb} A_{\ga \gb, n}= 0.$
 Hence, we have boundary $C^2$ estimates in
 each half ball $\overline {B}^+_r$ in Fermi coordinates.
 Thus,
  higher order regularities. 
 It remains to derive  a priori $C^0$ estimates.
 We begin by proving the boundedness of the integral term in
 (\ref{e:defm}).

 \begin{lemma} \label{l:int}
 Let $u$ be a solution to (\ref{e:defm}) with $t \in [0,1].$
 Then
  $ (1-t) (\int_M e^{- 5u})^{\frac{2}{5}} < C.$
 \end{lemma}
 \bpf
  Since
 $\frac{\de u}{\de n} =0$ on the boundary,  at the maximum point $x_0,$
 we  have $\gra u =0$ and  $\hess u$
 is negative semi-definite, no matter  
   $x_0$ being in the interior or at the boundary.
 Thus,
 $$
  (1-t) (\int_M e^{- 5u})^{\frac{2}{5}} \leq
 \gs_2 ^{\frac{1}{2}} (\hess u(x_0)  + A_g (x_0)+S_g(x_0))
  \leq \gs_2 ^{\frac{1}{2}} (A_g(x_0)+S_g(x_0)) < C.
  $$
 \epf
 Now we  prove that $\inf_M u > -C.$
 
 (1) $\inf_M u > -C$ for $t \in [0, 1-\varepsilon].$
 
   The analysis depends on whether the
    infimum point is close to the boundary or not.
    Suppose there is a sequence of solutions $\{u^i\}$ to (\ref{e:defm})
    with $t= t_i \leq 1- \varepsilon$ such
   that $u^i(p_i) = \inf u^i \rightarrow - \infty$ and $p_i \rightarrow p_0.$
    Let $\epsilon_i = e^{\inf u^i} \rightarrow +0$ and $d_i$ be the
   distance from $p_i$ to $\de M.$ We will show that there is a contradiction.\par

     Case a. Non-tangential approach.  Assume
    $
    \frac{d_i}{\epsilon_i} \rightarrow +\infty.
    $\\    
    Using the normal coordinates at $p_i$, we define the mapping
     $$\begin{array}{rl}
     \mathcal{T}_i : B(0, \frac{d_i}{\epsilon_i}) \subset \mathbb{R}^4 \rightarrow&
     M\\
       x \rightarrow& \exp_{p_i} (\epsilon_i x) = y.
     \end{array}$$ 
   On $\mathbb{R}^4$, define the metric $g_i = \epsilon_i^{-2} \mathcal{T}^*_i  g$
   and the function $\tilde{u}^i = u^i (\mathcal{T}_i(x)) - \ln \epsilon_i.$
   Then $\tilde{u}^i (0) = u^i(p_i) - \ln \epsilon_i = 0$ and
   $\tilde{u}^i(x) \geq 0$. Moreover, 
      $$\gs_2 ^{\frac{1}{2}} (\hess_{g_i} \tilde{u}^i +
   d\tilde{u}^i \otimes_{g_i} d\tilde{u}^i -
   \frac{1}{2}|\gra_{g_i} \tilde{u}^i|^2 g_i + A_{g_i}+ S_{g_i})
   = \epsilon^2_i (1-t_i) (\int_M e^{- 5u^i})^{\frac{2}{5}} + \zeta(t_i)
   f(\mathcal{T}_i(x))\, e^{-2\tilde{u}^i}$$ on $B(0,
   \frac{d_i}{\epsilon_i})$ in $\mathbb{R}^4.$ Note that $g_i$
   tends to the Euclidean metric $ds^2$
    as $i$ goes to infinity.
  By Lemma~\ref{l:int}, the integral term  is
  bounded.  Hence, by  local estimates \cite{Chen05} and the fact that
   $\tilde{u}^i  \geq 0$, we get
   $ \sup_{B(0,r)}  |\gra_{g_i} \tilde{u}^i| < C(r).$
   Integrating from zero, we have
   $
   \sup_{B(0,1)} \tilde{u}^i < C.$
  On the other hand, since $t \leq 1- \varepsilon$ for
  a fixed number $\varepsilon$, by Lemma~\ref{l:int}
  $$ \int_{B(0,1)} e^{- 5\tilde{u}^i} dV_{g_i}=
  \epsilon_i \int_{B(p_i,\epsilon_i)} e^{- 5u^i} dV_g
  < \epsilon_i C \rightarrow 0$$
  as $i \rightarrow \infty.$ This contradicts to $\sup_{B(0,1)} \tilde{u}^i < C$. 

  Case b. Tangential approach.
    Assume
     $\frac{d_i}{\epsilon_i} \leq C_0$
     for some fixed number $C_0$.\\
      Let $p'_i$ be a point on the boundary such
     that the distance between $p_i$ and $p'_i$ is $d_i.$
    We may assume the Fermi
    coordinates are defined in a tubular neighbor of length $\kappa.$
     Around the point $p'_i,$
    we define the mapping
     $$\begin{array}{rl}
     U_i : \overline{B}_+ (0, \frac{\kappa}{\epsilon_i}) \subset
     \overline {\mathbb{R}}_+^4 \rightarrow&  M\\
       x \rightarrow& \ G_{p'_i} (\epsilon_i x) = y,
     \end{array}$$ where $G$ is the normal exponential map;
      see  \cite{doCar}.
    We may assume that $\frac{d_i}{\epsilon_i} < \frac{\kappa}{\epsilon_i}.$
   On $\overline{\mathbb{R}}_+^4$, define the metric $g_i = \epsilon_i^{-2} U^*_i  g$
   and the function $\tilde{u}^i = u^i (U_i(x)) - \ln \epsilon_i.$
   Let $q_i \in \overline{B}_+ (0, \frac{\kappa}{\epsilon_i})$ be the point
   satisfying $U_i(q_i) = p_i.$ Therefore, $q_i \in \overline{B}_+ (0, \frac{d_i}{\epsilon_i})
   \subset \overline{B}_+ (0, C_0)$ belongs to a compact subset in
   $\overline{\mathbb{R}}_+^4.$ We have $\tilde{u}^i (q_i) = u^i(p_i) - \ln \epsilon_i = 0$ and
   $\tilde{u}^i(x) \geq 0$.  Moreover,
      $$\gs_2 ^{\frac{1}{2}} (\hess_{g_i} \tilde{u}^i +
   d\tilde{u}^i \otimes_{g_i} d\tilde{u}^i -
   \frac{1}{2}|\gra_{g_i} \tilde{u}^i|^2 g_i + A_{g_i}+ S_{g_i})
   = \epsilon^2_i (1-t_i) (\int_M e^{- 5u^i})^{\frac{2}{5}} + \zeta(t_i)
   f(U_i(x))\, e^{-2\tilde{u}^i}$$
  on $\overline{B}_+(0, \frac{\kappa}{\epsilon_i})$ in $\overline{\mathbb{R}}_+^4$
  with $\frac{\de \tilde{u}^i}{\de n} = 0$ on
  $\overline{B}_+(0, \frac{\kappa}{\epsilon_i}) \cap \{x_4 = 0\}.$  
   
  Using Theorem~\ref{t:bdy1} (a), Lemma~\ref{l:int}  and the fact that
   $\tilde{u}^i  \geq 0$, we get
   $ \sup_{\overline{B}_+(0,r)}  |\gra_{g_i} \tilde{u}^i| < C(r).$
    Integrating from $q_i$, we have
  $
   \sup_{\overline{B}_+(0,C_0)} \tilde{u}^i < C.$
  On the other hand, since $t \leq 1- \varepsilon,$ by Lemma~\ref{l:int}
  $$ \int_{\overline{B}_+(0,C_0)} e^{- 5\tilde{u}^i} dV_{g_i}=
  \epsilon_i \int_{U_i (\overline{B}_+(0,C_0))} e^{- 5u^i} dV_g
  < \epsilon_i C \rightarrow 0$$
  as $i \rightarrow \infty.$ This contradicts to $\sup_{\overline{B}_+(0,C_0)} \tilde{u}^i < C$. 

 (2) $\inf_M u > -C$ when $t \rightarrow 1.$

   Suppose on the contrary there is a sequence of solutions $\{u^i\}$ with $t_i \rightarrow 1$
   such that $u^i(p_i) = \inf u^i \rightarrow - \infty$ and $p_i
   \rightarrow p_0.$
    Let $\epsilon_i = e^{\inf u^i} \rightarrow +0$ and $d_i$ be the
   distance from $p_i$ to $\de M.$ For simplicity, we denote 
   $e^{-2u^i} g$ by $\hat g_i$ and $A_{\hat g_i}$ by
   $\hat A_i.$\par

  Case a. Non-tangential approach.
   Assume
   $
    \frac{d_i}{\epsilon_i} \rightarrow +\infty.
   $\\
   Let $\mathcal{T}_i$, $g_i,$ and $\tilde{u}^i$  be as in (1) Case a.
   Denote the metric $e^{-2\tilde{u}^i} g_i$ by $\tilde{g_i}.$
   Then $\tilde{u}^i (0) = u^i(p_i) - \ln \epsilon_i = 0$ and
   $\tilde{u}^i(x) \geq 0$. Moreover, since $t_i \rightarrow 1,$
   we have $\zeta(t_i) =1.$ Therefore, 
      $$\gs_2 ^{\frac{1}{2}} (\hess_{g_i} \tilde{u}^i +
   d\tilde{u}^i \otimes_{g_i} d\tilde{u}^i -
   \frac{1}{2}|\gra_{g_i} \tilde{u}^i|^2 g_i+ A_{g_i})
   = \epsilon^2_i (1-t_i) (\int_M e^{- 5u^i})^{\frac{2}{5}} +
   f(\mathcal{T}_i(x))\, e^{-2\tilde{u}^i}$$
   on $B(0, \frac{d_i}{\epsilon_i})$ in $\mathbb{R}^4.$
    Similar to  (1) Case a,   
   we get
      $ \sup_{B(0,r)} |\tilde{u}^i|+ |\gra_{g_i} \tilde{u}^i|^2 +
   |\hess_{g_i} \tilde{u}^i|< C(r).$

   Now since $f(\mathcal{T}_i (x)) \rightarrow f(p_0),$ the equation
   is uniform elliptic and concave.    Notice that $B(0, \frac{d_i}{\epsilon_i})\rightarrow \mathbb{R}^4$.
   Therefore,  $\{\tilde{u}^i\}$ converges uniformly on compact sets to a solution
   $u \in C^{\infty}(\mathbb{R}^4)$ of
   $\gs_2 ^{\frac{1}{2}} (\hess u + du \otimes u -
  \frac{1}{2}|\gra u|^2 ds^2 ) = f(p_0)\,
  e^{-2u}.$ 
   By the uniqueness theorem  \cite{CGY02a},  $e^{-2u} ds^2$
   comes from the pulling-back of the standard metric $g_c$ on the sphere. Hence,
   $$ 4\pi^2 \leftarrow \int_{B{(0,\frac{d_i}{\epsilon_i})}} \gs_2
    (A_{\tilde{g}_i})dV_{\tilde{g}_i} = \int_{B( p_i, d_i)}
    \gs_2(\hat A_i) dV_{\hat g_i} \leq \int_M \gs_2 (\hat A_i) dV_{\hat g_i}. $$
   On the other hand, since $\mu_{\hat g_i}= 0,$ by
   Lemma~\ref{l:b=0} we have $ \mathcal{B}_{\hat g_i}= 0.$ Thus, by
   Lemma~\ref{l:conf}
   $ \int_M \gs_2 (\hat A_i) dV_{\hat g_i}= \int_M \gs_2 (\hat A_i) dV_{\hat g_i}
   + \frac{1}{2} \oint_{\de M}  \mathcal{B}_{\hat g_i}\, d\Sigma_{\hat g_i} \leq 2 \pi^2.$
   This gives  a contradiction. \par


   Case b. Tangential approach.
     Assume
     $\frac{d_i}{\epsilon_i} \leq C_0$
     for some fixed number $C_0$. \\
     Let $p'_i$  and  $\kappa$ be as in (1) Case b.
   We may assume that
   $\frac{d_i}{\epsilon_i} < \frac{\kappa}{\epsilon_i}.$
   Let $U_i$, $g_i, q_i$ and $\tilde{u}^i$ be also as in (1) Case b.
   Denote the metric $e^{-2\tilde{u}^i} g_i$ by $\tilde{g}_i.$
     $q_i \in \overline{B}_+ (0, \frac{d_i}{\epsilon_i})
   \subset \overline{B}_+ (0, C_0)$ belongs to a compact subset in
   $\overline{\mathbb{R}}_+^4.$ We have $\tilde{u}^i (q_i) = u^i(p_i) - \ln \epsilon_i = 0$ and
   $\tilde{u}^i(x) \geq 0$. Moreover, 
      $$\gs_2 ^{\frac{1}{2}} (\hess_{g_i} \tilde{u}^i +
   d\tilde{u}^i \otimes_{g_i} d\tilde{u}^i -
   \frac{1}{2}|\gra_{g_i} \tilde{u}^i|^2 g_i+ A_{g_i})
   = \epsilon^2_i (1-t_i)(\int_M e^{- 5u^i})^{\frac{2}{5}} +
   f(U_i(x))\, e^{-2\tilde{u}^i}$$
  on $\overline{B}_+(0, \frac{\kappa}{\epsilon_i})$ in $\overline{\mathbb{R}}_+^4$
  with $\frac{\de \tilde{u}^i}{\de n} = 0$ on
  $\overline{B}_+(0, \frac{\kappa}{\epsilon_i}) \cap \{x_4 = 0\}.$  
    Similar to  (1) Case b,
  by Theorem~\ref{t:bdy1} (a), we get
      $ \sup_{\overline{B}_+(0,r)}  |\tilde{u}^i|+ |\gra_{g_i} \tilde{u}^i| < C(r).$
   Then by Theorem~\ref{t:bdy1} (b), we arrive at
   $ \sup_{\overline{B}_+(0,r)}  |\tilde{u}^i|+ |\gra_{g_i} \tilde{u}^i|+
   |\hess_{g_i} \tilde{u}^i| < C(r).$

   Now 
    $\{\tilde{u}^i\}$ converges uniformly on compact sets to a solution
   $u \in C^{\infty}(\overline{\mathbb{R}}_+^4)$ of
   $\gs_2 ^{\frac{1}{2}} (\hess u + du \otimes u -
  \frac{1}{2}|\gra u|^2 ds^2 ) = f(p_0)\,
  e^{-2u}$
  with $\frac{\de u}{\de n} = 0$ on $\{x_4 = 0\}.$  By reflection, $u$  extends to a $C^{2,\ga}$
   solution of the above equation in $\mathbb{R}^4$.
   Further regularities give  $u \in C^{\infty} (\mathbb{R}^4).$
   By the uniqueness theorem,  $e^{-2u} ds^2$ 
    comes from the pulling-back of  $g_c.$ 
   Hence,
   $$ 2\pi^2 \leftarrow \int_{\overline{B}_+(0,\frac {\kappa}{\epsilon_i})} \gs_2
    (A_{\tilde{g}_i})dV_{\tilde{g}_i} = \int_{U_i (\overline{B}_+(0, \frac{\kappa}{\epsilon_i}))}
    \gs_2(\hat A_i) dV_{\hat g_i} \leq \int_M \gs_2 (\hat A_i) dV_{\hat g_i}. $$
   On the other hand, since $\mu_{\hat g_i} = 0,$ by
   Lemma~\ref{l:b=0} we have $\mathcal{B}_{\hat g_i} = 0.$ Thus, by
   Lemma~\ref{l:conf} and the assumption that $( M, g)$ is not
   conformally equivalent to the hemisphere, we finally arrive at
   $ \int_M \gs_2 (\hat A_i) dV_{\hat g_i}= \int_M \gs_2 (\hat A_i) dV_{\hat g_i}
   + \frac{1}{2} \oint_{\de M}  \mathcal{B}_{\hat g_i}\, d\Sigma_{\hat g_i} < 2 \pi^2.$
   This gives a contradiction. \par


  (3) $C^0$ estimates.

  Once $u$ has a lower bound, by \cite{Chen05} and
  Theorem~\ref{t:bdy1} (a) we have $|\gra u|< C$.
  Thus, we obtain  $\sup_M u \leq \inf_M u + C.$
  It remains to prove that $\inf_M u$ is upper bounded.
  Since
 $\frac{\de u}{\de n} =0$ on the boundary,  at the minimun point $x_0,$
 we  have $\gra u =0$ and  $\hess u$
 is positive semi-definite, no matter  
   $x_0$ being in the interior or at the boundary.
 Therefore,
  $$
  C e^{- 2 \inf u}  \geq
  (1-t) (\int_M e^{- 5u})^{\frac{2}{5}} + \zeta(t) f(x_0) 
  e^{-2u}
  = \gs_2 ^{\frac{1}{2}} (\hess u(x_0)+ A_g (x_0))
  \geq \gs_2 ^{\frac{1}{2}} (A_g(x_0)) > 0.
  $$
  \end{proof}

\subsection{Application to Einstein manifolds}\label{subs:ccem}

 In this subsection, we give an application of Theorem~\ref{t:main} to conformally compact Einstein manifolds. 
 \begin{definition}
 Let $X^4$ be a compact manifold with boundary $\partial X = N^3$ and
 $g$ be a complete Einstein metric defined in the interior of $X$.
 $(X,g)$ is called a conformally compact Einstein manifold
 if there exists a smooth defining function $s$ for $N$ such that $(X, s^2g)$
 is a compact Riemannian manifold with boundary.
\end{definition}
 Each defining function induces a metric $s^2 g |_{N} = g_0$ on $N$.
 Thus $(X,g)$ determines a conformal structure $(N^3, [g_0])$ called
 the \emph{conformal infinity}. 
  The \emph{renormalized volume} $\mathcal{V}$
is a invariant  of $(X, g)$ coming from the volume expansion
  $Vol(\{s> \epsilon \}) = c_0 \epsilon^{-3}+ c_2 \epsilon^{-1} +
  \mathcal{V} + o(1).$

 \begin{corollary} \label{c:ccem}
 Let $(X^4,g)$ be a conformally compact Einstein manifold with conformal
 infinity $(N^3, [g_0])$. Suppose that $Y(N^3, [g_0])$
 and the renormalized volume $\mathcal{V}$ are both positive.
  Then there exists a conformal compactification $(X, \rho^2 g)$
  such that  $\gs_2(A_{\rho^2 g})$  is a positive constant and
  the boundary is totally geodesic.
  Moreover, $\rho$ is a defining function for $N.$
 \end{corollary}

 \begin{proof}
  First,  Qing \cite{Q03}, \cite{CQY04} proved
  that if $Y(N^3, [g_0]) > 0,$ then there exists a conformal compactification
  $(X, e^{-2u}g)$ such that $R_{e^{-2u}g}$ is positive and the boundary
  is totally geodesic. Denote this metric by $g_1 = e^{-2u}g.$
  Hence, we have
  \beq \label{i:yamabe}
  Y(X, N, [g_1])> 0.
  \eeq
   Secondly, for conformally compact Einstein
  four-manifolds, Andersen \cite{And01} proved that 
  $
  32 \pi^2 \chi (X) =  \int_X |\mathcal{W}|^2 dV_g + 4
  \mathcal{V}.
  $ Now recall the Gauss-Bonnet formula for compact four-manifolds
  with boundary:
   $ 32 \pi^2 \chi (X, \de X) = \int_X |\mathcal{W}|^2 + 16 (\int_X \gs_2 (A_{g_1}) + \frac{1}{2}
   \oint_{\de X} \mathcal{B}_{g_1}).$
  Since the boundary is totally
  geodesic $h_{g_1} = 0,$  by Lemma~\ref{l:b=0} we have
  $\mathcal{B}_{g_1} = 0.$ 
  This gives
  \beq \label{i:renom}
  4 \int_X \gs_2 (A_{g_1})  = \mathcal{V} > 0.
  \eeq

  (\ref{i:yamabe}) and (\ref{i:renom}) then verify the conditions of Theorem~\ref{t:main}.
  Therefore, by Theorem~\ref{t:main} there is a conformal metric $g_2 = e^{-2v} g_1$ such
  that $\gs_2 (A_{g_2})$ is a positive constant and the boundary
  is totally geodesic. Thus, $(X, g_2= \rho^2 g)$ with $\rho = e^{-(u+ v)}$ is a conformal
  compactification satisfying the properties required in the
  corollary. Moreover, since $e^{-u}$ is a defining function (see
  \cite{Q03}), it follows that $e^{-(u+v)}$ is also a defining
  function.
 \end{proof}

\section{Functionals $\mathcal{F}_k$}  \label{s:variation}

 In this section, we prove Theorem~\ref{t:variation} and
 Corollary~\ref{c:variation}. We first prove a lemma.
 \begin{lemma}\label{l:divergence} Let $A_g$ and $L$ be the Schouten tensor and the
 second fundamental form, respectively.
  When the Cotten tensor is zero (i.e., $A_{ij,k}= A_{ik,j}$) or when $q=1$,
  we have\par
 (a) $T_q(A)^i_{j, i} =0.$ (i.e., $T_q$ is divergence free.) \\
  Moreover, if $(M, g)$ is locally conformally flat, we also
  have\par
 (b) $T_{q, r} (A^T, L)^{\ga}_{\gb, \tilde{\ga}}= \frac{(n-1-q)(q-r)}{q}
 T_{q-1,r} (A^T, L)^{\ga}_{\gb}A^n_{\ga}- r T_{q, r-1}(A^T, L)^{\ga}_{\gb}
 A^n_{\ga};$\par
 (c) $T_q(A^T)^{\ga}_{\gb, \tilde{\ga}}= -q T_{q, q-1} (A^T,
 L)^{\ga}_{\gb} A^n_{\ga}.$
 \end{lemma}
 \bpf When $(M, g)$ is locally conformally flat, (a) was proved in \cite{Via00}; see also \cite{CGY02} for
 $q=1$ case. 
  Suppose (a) is true for $q <m.$ By the recursive formula and $A_{ij,k}= A_{ik,j}$,
  $$T_m (A)^i_{j, i} = \gs_m (A)_{, i} g^i_j - T_{m-1}(A)^i_{k, i} A^k_j - T_{m-1}(A)^i_k A^k_{j,i}
   = \gs_m (A)_{, j}- T_{m-1}(A)^i_k A^k_{j,i}= 0.$$

 For (b), we first compute
 \begin{eqnarray*}
  T_{q, r} (A^T, L)^{\ga}_{\gb, \ga}&=& \frac{1}{q!}
  \sum_{i_1, \cdots, j_1 \cdots <n}
   \left(  \begin{array}{l}
    i_1 \cdots i_r \cdots i_q \, \ga\\
    j_1 \cdots j_r \cdots j_q \, \gb
    \end{array} \right)\times\\
   & &   \left[ r A^{j_1}_{i_1}\cdots A^{j_r}_{i_r, \ga}
    L^{j_{r+1}}_{i_{r+1}} \cdots L^{j_q}_{i_q}
    + (q-r) A^{j_1}_{i_1}\cdots A^{j_r}_{i_r}
    L^{j_{r+1}}_{i_{r+1}} \cdots L^{j_q}_{i_q, \ga}\right]\\
    &=& \frac{1}{q!}
  \sum_{i_1, \cdots, j_1 \cdots <n} (q-r)
   \left(  \begin{array}{l}
    i_1 \cdots i_r \cdots i_q \, \ga\\
    j_1 \cdots j_r \cdots j_q \, \gb
    \end{array} \right)
     A^{j_1}_{i_1}\cdots A^{j_r}_{i_r}
    L^{j_{r+1}}_{i_{r+1}} \cdots L^{j_q}_{i_q, \ga},
  \end{eqnarray*}
  where in the first equality, the first term is zero because
  $A^{j_r}_{i_r, \ga}$ is symmetric in $(i_r \ga).$
  By the Codazzi equation
  and the curvature decomposition, we have
  $L_{\ga \gc, \gb} -L_{\gb \gc, \ga} = R_{\ga \gb \gc n} = A_{\gb n} g_{\ga \gc}
   - A_{\ga n} g_{\gb \gc}.$ Therefore, 
  $$\begin{array} {l}
   T_{q, r} (A^T, L)^{\ga}_{\gb, \ga}
    = \frac{q-r}{q!}  \sum_{i_1, \cdots, j_1 \cdots <n}
   \left(  \begin{array}{l}
    i_1 \cdots i_q \, \ga\\
    j_1 \cdots j_q \, \gb
    \end{array} \right)
     A^{j_1}_{i_1}\cdots A^{j_r}_{i_r}
    L^{j_{r+1}}_{i_{r+1}} \cdots L^{j_{q-1}}_{i_{q-1}} g^{j_q}_{i_q} A_{\ga
    n}.
       \end{array} $$
 Hence,
  \beq   T_{q, r} (A^T, L)^{\ga}_{\gb, \ga}
   = \frac{q-r}{q!} (n-q-1)\, T_{q-1, r}(A^T, L)^{\ga}_{\gb} A_{\ga
   n}.\label{e:S3l1}
 \eeq

 By definition,
   $\gra_{\gc} A^{\gb}_{\ga} = \gra_{\tilde{\gc}} A^{\gb}_{\ga}- L_{\ga \gc} A^{\gb}_n
   - L^{\gb}_{\gc} A^n_{\ga}.$ Thus, we obtain
  \begin{align}
   &  T_{q, r} (A^T, L)^{\ga}_{\gb, \tilde{\ga}}
    = T_{q, r} (A^T, L)^{\ga}_{\gb, \ga} \notag\\
   & + \frac{1}{q!}  \sum_{i_1, \cdots, j_1 \cdots <n}
   \left(  \begin{array}{l}
    i_1 \cdots i_r \cdots i_q \, \ga\\
    j_1 \cdots j_r \cdots j_q \, \gb
    \end{array} \right)
    r A^{j_1}_{i_1}\cdots A^{j_{r-1}}_{i_{r-1}} (L_{i_r \ga} A^{j_r}_n + L^{j_r}_{\ga} A^n_{i_r})
    L^{j_{r+1}}_{i_{r+1}} \cdots L^{j_q}_{i_q}\notag\\
   &  = T_{q, r} (A^T, L)^{\ga}_{\gb, \ga} +
    \frac{r}{q!}  \sum_{i_1, \cdots, j_1 \cdots <n}
   \left(  \begin{array}{l}
    i_1 \cdots i_q \, \ga\\
    j_1 \cdots j_q \, \gb
    \end{array} \right)
    A^{j_1}_{i_1}\cdots A^{j_{r-1}}_{i_{r-1}}L^{j_r}_{\ga} A^n_{i_r}
    L^{j_{r+1}}_{i_{r+1}} \cdots L^{j_q}_{i_q},\notag
   \end{align}
   where in the first equality, the first term is zero because
  $L$ is symmetric.
   Exchanging $i_r$ and $\ga,$ we arrive
   at
   \begin{align}
     &  T_{q, r} (A^T, L)^{\ga}_{\gb, \tilde{\ga}}
        = T_{q,r} (A^T, L)^{\ga}_{\gb, \ga} - rT_{q, r-1} (A^T,
   L)^{i_r}_{\gb} A^n_{i_r}. \label{e:S3l2}
  \end{align}
  Combining (\ref{e:S3l1}) and (\ref{e:S3l2}) gives (b).

 (c) follows from (b) by letting $r= q.$
 \epf

 In the
 following proof, for simplicity $\int$ stands for $\int_M$ and
 $\oint$ stands for $\oint_{\de M}.$
 \bpf[Proof of Theorem~\ref{t:variation}]
  Let $g_t= e^{-2u_t} g$ be a conformal variation of $g$ such that $u_0= 0.$ Suppose
  $u'_t= \phi$ at $t=0.$ Then $g'= -2 \phi g$ and $(g^{-1})'= 2\phi
  g^{-1}.$  Consequently,  $dV' = -n \phi dV$ and $d\Sigma' = - (n-1) \phi d\Sigma.$

  By conformal change formulas of $A_g$ and $L,$ we get directly
  that $A'_{ij} = \phi_{ij}$ and $L'_{\ga \gb}= -L_{\ga \gb} \phi + \phi_n g_{\ga \gb}.$
  Therefore, by raising indices we obtain
  \beq\label{e:S3A}
  {A'}^j_i = A'_{im} g^{mj}+ A_{im} g'^{mj} = \phi^j_i + 2\phi
  A^j_i
  \eeq and
  \beq \label{e:S3L}
  {L'}^{\gb}_{\ga}= L'_{\ga \gc} g^{\gc \gb}+ L_{\ga \gc} g'^{\gc \gb}=
  L^{\gb}_{\ga} \phi + \phi_n g^{\gb}_{\ga}.
  \eeq Then by Lemma~\ref{l:variation},  we have
  \begin{align}
  \gs'_{q+1, r+1} (A^T, L)&= (r+q+2) \gs_{q+1, r+1} (A^T, L)\phi +
  \frac{r+1}{q+1}T_{q, r} (A^T, L)^{\ga}_{\gb} \phi^{\gb}_{\ga}\notag\\
  &  + \frac{(q-r)}{q+1}(n-1-q) \gs_{q, r+1}(A^T, L) \phi_n, \label{e:var-AtL}\\
   \gs'_{q+1} (A)& = 2(q+1) \gs_{q+1} (A)\phi + T_q (A)^i_j
  \phi^j_i,\label{e:var-A}\\
  \gs'_{q+1} (L) &= (q+1) \gs_{q+1} (L) \phi + (n-1-q)\gs_q(L)
  \phi_n.\label{e:var-L}
  \end{align}

  By Lemma~\ref{l:divergence} (a), $T_q(A)$ is divergence free.
  Applying the integration
  by parts gives
  \beq \label{e:S3}
    \left(\int \gs_k (A)\right)' = \int \gs_k'(A) dV + \gs_k (A)
    dV'= (2k-n)\int \gs_k(A) \phi - \oint T_{k-1} (A)^n_j \phi^j,
  \eeq where $n$ is the unit inner normal.

 (a)
 By (\ref{e:var-AtL}) and Lemma~\ref{l:Tij} (c),
  \begin{align}
  & \left( \oint \gs_{2,1} (A^T, L) \right)'=
  \oint  \{ (4-n) \phi \gs_{2, 1} (A^T, L) +
  \frac{1}{2} T_1 ( L)^{\ga}_{\gb}
  \phi^{\gb}_{\ga} + \frac{n-2}{2} \gs_1 (A^T) \phi_n\} \notag \\
  &= \oint  \{ (4-n) \phi \gs_{2, 1} (A^T, L) +
  \frac{1}{2} T_1 ( L)^{\ga}_{\gb}
  \phi^{\gb}_{\ga} + \frac{n-2}{2} T_1 (A)^n_n \phi_n
  \}.\label{e:S3a1}
  \end{align}
  For the second term in the last integral, applying integration by
  parts we get
  \begin{align}
  & \oint T_1 ( L)^{\ga}_{\gb}  \phi^{\gb}_{\ga}
   = \oint T_1 (L)^{\ga}_{\gb} (\phi^{\tilde{\gb}}_{\tilde{\ga}}- L^{\gb}_{\ga} \phi_n)
   = \oint \{- T_1 (L)^{\ga}_{\gb, \ga} \phi^{\gb}
    - 2\gs_2 (L) \phi_n\},\label{e:S3a2}
   \end{align}
  where in the last equality we use Lemma~\ref{l:variation}(a), and the
  fact that $L_{\ga \gb, \gc} = L_{\ga \gb, \tilde{\gc}}$ since
  the boundary is of codimension one. On the other hand, by the
  Codazzi equation, we have  $R_{\gb n} = -L^{\gc}_{\gb, \gc} + h_{,\gb}.$
  Therefore, we get $T_1 (A)_{\gb}^n= - A_{\gb}^n = -\frac{1}{n-2} R_{\gb}^n
  = \frac{1}{n-2} (L^{\gc}_{\gb, \gc} - h_{,\gb}).$ As a result,
  we have the relation $T_1(L)^{\ga}_{\gb, \ga} =
  h_{,\ga} g^{\ga}_{\gb} - L^{\ga}_{\gb, \ga} =  h_{,\gb} - L^{\ga}_{\gb, \ga} =
  - (n-2) T_1(A)_{\gb}^n.$ Combining this relation, (\ref{e:S3a1})
  and (\ref{e:S3a2}) gives
   \beq
   \left( \oint \gs_{2,1} (A^T, L) \right)'
   = \oint  \{ (4-n) \phi \gs_{2, 1} (A^T, L) +
  \frac{n-2}{2} T_1 ( A)^{n}_j
  \phi^j - \gs_2(L) \phi_n
  \}.
  \eeq

  For $n > 4,$ using (\ref{e:var-L}) we have
   $
   \left( \oint \gs_3 (L) \right)' =
   \oint \{(4-n) \gs_3 (L) \phi + (n-3) \gs_2 (L) \phi_n\}.
   $
   Recall that $\mathcal{B}^2 = \frac{2}{n-2} \gs_{2,1} (A^T, L)+
   \frac{2}{(n-2)(n-3)} \gs_3(L).$ Hence, we obtain
   $\left(\oint \mathcal{B}^2\right)' = (4-n) \oint \mathcal{B}^2 \phi +
   \oint T_1 (A)^n_j \phi^j.$ Going back to (\ref{e:S3}), we
   finally arrive at
   \begin{eqnarray*}
   & &\left(\int \gs_2(A)dV + \oint \mathcal{B}^2 d\Sigma - \Lambda \int dV
   \right)'
   =  (4-n) \left( \int \gs_k \phi + \oint \mathcal{B}^2
   \phi \right) - \Lambda \int n\phi
  \end{eqnarray*} for constant $\Lambda.$
  Since $n-4\neq 0,$ critical points of $\mathcal{F}_2$
  restricted on $\mathcal{M}$ satisfy $\gs_2 = constant$  in $M$ and $\mathcal{B}^2 = 0$ in
  $\de M.$

   For $n = 3,$  note that $\frac{1}{3} h^3 - \frac{1}{2} h |L|^2
   = - \frac{1}{6} \gs_1 (L)^3 + \gs_1(L) \gs_2 (L).$
   Then by  (\ref{e:var-L}), we have
   \begin{eqnarray*}
   \left( \oint - \frac{1}{6} \gs_1 (L)^3 + \gs_1(L) \gs_2 (L) \right)'
   &=& \oint \{(- \frac{1}{6} \gs_1 (L)^3 + \gs_1(L) \gs_2 (L)) \phi + 2 \gs_2 (L) \phi_n\}.
   \end{eqnarray*}
   Recall that $\mathcal{B}^2 = 2 \gs_{2,1} (A^T, L)+
   \frac{1}{3} h^3 - \frac{1}{2} h |L|^2.$ Hence, we obtain
   $\left(\oint \mathcal{B}^2\right)' =  \oint \mathcal{B}^2 \phi +
   \oint T_1 (A)^n_j \phi^j.$ Now the rest of proof is the same
   as $n> 4$ case.

 (b)
   By (\ref{e:var-AtL}),
  \begin{align}
  & \left( \oint \gs_{2k-i-1,i} (A^T, L) \right)'   = \oint  \{ (2k-n) \phi \gs_{2k-i-1, i} (A^T, L) 
  \notag\\
  &+\frac{i}{2k-i-1} T_{2k-i-2, i-1} (A^T, L)^{\ga}_{\gb}
  \phi^{\gb}_{\ga}+ \frac{2k-2i-1}{2k-i-1} (n-2k+i+1) \gs_{2k-i-2, i} (A^T,L) \phi_n
  \}.\label{e:S3b1}
  \end{align}

  For the second term in the last integral, applying integration by
  parts we have
  \begin{align}
  & \oint T_{2k-i-2, i-1} (A^T, L)^{\ga}_{\gb}  \phi^{\gb}_{\ga}
   = \oint \{- T_{2k-i-2, i-1} (A^T, L)^{\ga}_{\gb, \tilde{\ga}} \phi^{\gb}
    - T_{2k-i-2, i-1} (A^T, L)^{\ga}_{\gb} L^{\gb}_{\ga}
    \phi_n\}\notag\\
    &= \oint \{- \frac{(n-2k+i+1)(2k-2i-1)}{2k-i-2}T_{2k-i-3, i-1} (A^T,L)^{\ga}_{\gb}A^n_{\ga}
    \phi^{\gb}\notag\\
    &  + (i-1) T_{2k-i-2, i-2} (A^T, L)^{\ga}_{\gb} A^n_{\ga}
    \phi^{\gb}- (2k-i-1) \gs_{2k-i-1, i-1} (A^T, L) \phi_n\},
    \label{e:S3b2}
  \end{align}
  where in the last equality we use Lemma~\ref{l:divergence}(b) and
  Lemma~\ref{l:variation}(a).

  Now recall that $\mathcal{B}^k = \sum_{i=0}^{k-1} C_1(n,k,i) \gs_{2k-i-1, i}.$
  Combining (\ref{e:S3b1}) and (\ref{e:S3b2}) gives
  $$\left(\oint \mathcal{B}^k\right)' = (2k-n) \oint \mathcal{B}^k \phi +
  \oint I* A^n_{\ga} \phi^{\gb} + \oint II * \phi_n,$$
  where \\
  $I = \sum_{i=0}^{k-1} C_1 (n,k,i)
  [- \frac{(n-2k+i+1)(2k-2i-1)i}{(2k-i-1)(2k-i-2)} T_{2k-i-3, i-1}
  (A^T,L)^{\ga}_{\gb}+ \frac{i(i-1)}{2k-i-1} T_{2k-i-2, i-2} (A^T,
  L)^{\ga}_{\gb}]$ and $II= \sum_{i=0}^{k-1} C_1(n,k,i)
  [-i \gs_{2k-i-1, i-1} (A^T, L)+ \frac{(2k-2i-1)(n-2k+i+1)}{2k-i-1} \gs_{2k-i-2, i} (A^T,L)].$
  By definition, we have $C_1 = \frac{(2k-i-1)!(n-2k+i)!}{(n-k)!(2k-2i-1)!!\, i!}.$
  Straightforward computations yield
  \begin{eqnarray*}
  I&=& \sum_{i=1}^{k-1} - \frac{(2k-i-3)! (n-2k+i+1)!}{(n-k)! (2k-2i-3)!! (i-1)!}
  T_{2k-i-3, i-1}(A^T, L)^{\ga}_{\gb}\\
   &+& \sum_{i=2}^{k-1} \frac{(2k-i-2)!(n-2k+i)!}{(n-k)!
   (2k-2i-1)!!(i-2)!}T_{2k-i-2, i-2} (A^T, L)^{\ga}_{\gb}
    = - T_{k-2}
   (A^T)^{\ga}_{\gb},
  \end{eqnarray*}
  where the terms cancel out except the $i= k-1$
  term in the first summation. For II, 
  \begin{eqnarray*}
  II &=& \sum_{i=1}^{k-1} - \frac{(2k-i-1)! (n-2k+i)!}{(n-k)! (2k-2i-1)!! (i-1)!}
  \gs_{2k-i-1, i-1}(A^T, L)\\
   &+& \sum_{i=0}^{k-1} \frac{(2k-i-2)!(n-2k+i+1)!}{(n-k)!
   (2k-2i-3)!!\, i!} \gs_{2k-i-2, i} (A^T, L)= \gs_{k-1} (A^T),
  \end{eqnarray*}
  where all terms are cancelled except the $i= k-1$
  term in the second summation.

  Finally, using Lemma~\ref{l:Tij} (c) and (d) we obtain
  \begin{eqnarray*}
  \left(\oint \mathcal{B}^k\right)' &=& (2k-n) \oint \mathcal{B}^k \phi +
  \oint -T_{k-2}(A^T)^{\ga}_{\gb} A^n_{\ga} \phi^{\gb} + \oint \gs_{k-1}(A^T)
  \phi_n\\
   &=& (2k-n) \oint \mathcal{B}^k \phi  + \oint T_{k-1} (A)^n_j
   \phi^j.
  \end{eqnarray*}
  Hence, by (\ref{e:S3}) we arrive at
  \begin{eqnarray*}
   & &\left(\int \gs_k(A)dV + \oint \mathcal{B}^k d\Sigma - \Lambda \int dV
   \right)'
   =  (2k-n) \left( \int \gs_k \phi + \oint \mathcal{B}^k
   \phi \right) - \Lambda \int n\phi
  \end{eqnarray*} for constant $\Lambda.$
  Since $n-2k \neq 0,$  this gives the result.
  
 (c) First note that when the boundary is umbilic, by
 (\ref{e:S3L}) we have $\mu' = \mu \phi + \phi_n.$
 Therefore, by (\ref{e:var-AtL}) we have
 \begin{align}
  \left(\oint \gs_i (A^T) \mu^{2k-2i-1}\right)'
  &= \oint \{(2k-n) \gs_i (A^T) \mu^{2k-2i-1}\phi +
 T_{i-1} (A^T)^{\ga}_{\gb} \phi^{\gb}_{\ga}
 \mu^{2k-2i-1}  \notag\\
  &+ (2k -2i-1)\gs_i(A^T) \mu^{2k-2i-2} \phi_n \}. \label{e:S3c1}
 \end{align}

 For the second term in the last integral, applying integration by
  parts we have
  \begin{align}
  & \oint T_{i-1} (A^T)^{\ga}_{\gb}  \phi^{\gb}_{\ga} \mu^{2k-2i-1}
   = \oint T_{i-1} (A^T)^{\ga}_{\gb} (
   \phi^{\tilde{\gb}}_{\tilde{\ga}}- \mu g^{\gb}_{\ga} \phi_n) \mu^{2k-2i-1} \notag\\
      &= \oint \{ (n-i) T_{i-2} (A^T)^{\ga}_{\gb} \phi^{\gb} \mu^{2k-2i} \mu_{\ga}
    - (2k-2i-1)T_{i-1} (A^T)^{\ga}_{\gb} \mu_{\ga}
    \phi^{\gb}\notag\\
    &  - (n-i) \gs_{i-1} (A^T) \mu^{2k-2i} \phi_n\},
    \label{e:S3c2}
  \end{align}
  where in the last equality we use Lemma~\ref{l:divergence}(c) and
  Lemma~\ref{l:bdy}(a).

  Recall that $\mathcal{B}^k = \sum_{i=0}^{k-1} C_2(n,k,i) \gs_i \mu^{2k-2i-1}.$
  Combining (\ref{e:S3c1}) and (\ref{e:S3c2}) gives
  $$\left(\oint \mathcal{B}^k\right)' = (2k-n) \oint \mathcal{B}^k \phi +
  \oint I* \mu_{\ga} \phi^{\gb} + \oint II * \phi_n,$$
  where \\
  $I = \sum_{i=0}^{k-1} C_2 (n,k,i)
  [- (2k-2i-1) T_{i-1}(A^T)^{\ga}_{\gb}\mu^{2k-2i-2}+ (n-i) T_{i-2} (A^T)^{\ga}_{\gb}\mu^{2k-2i}]$
   and $II= \sum_{i=0}^{k-1} C_2(n,k,i)
  [-(n-i) \gs_{i-1} (A^T) \mu^{2k-2i}+ (2k-2i-1) \gs_i (A^T)\mu^{2k-2i-2}].$
  By definition, we have $C_2 = \frac{(n-i-1)!}{(n-k)!(2k-2i-1)!!}.$
  Straightforward computations yield
  \begin{eqnarray*}
  I&=& \sum_{i=1}^{k-1} - \frac{(n-i-1)!}{(n-k)! (2k-2i-3)!!}
  T_{i-1}(A^T)^{\ga}_{\gb} \mu^{2k-2i-2}\\
   &+& \sum_{i=2}^{k-1} \frac{(n-i)!}{(n-k)!
   (2k-2i-1)!!}T_{i-2} (A^T)^{\ga}_{\gb}\mu^{2k-2i}
   = - T_{k-2} (A^T)^{\ga}_{\gb},
  \end{eqnarray*}
  where all terms are cancelled except the $i= k-1$
  term in the first summation. For II ,
  \begin{eqnarray*}
  II &=& \sum_{i=1}^{k-1} - \frac{(n-i)!}{(n-k)! (2k-2i-1)!!}
  \gs_{i-1}(A^T)\mu^{2k-2i}\\
   &+& \sum_{i=0}^{k-1} \frac{(n-i-1)!}{(n-k)!
   (2k-2i-3)!!} \gs_i (A^T) \mu^{2k-2i-2}
   =   \gs_{k-1} (A^T),
  \end{eqnarray*}
  where all terms are cancelled except the $i= k-1$
  term in the second summation.

  Noting that by Lemma~{\ref{l:bdy}}, we have $A^n_{\ga} = \mu_{\ga}.$ As a result, we obtain
 $
  \left(\oint \mathcal{B}^k\right)' = (2k-n) \oint \mathcal{B}^k \phi +
  \oint -T_{k-2}(A^T)^{\ga}_{\gb} \mu_{\ga} \phi^{\gb} + \oint \gs_{k-1}(A^T)
  \phi_n.
  $
   By Lemma~\ref{l:Tij} (c) and (d), this gives $(\oint \mathcal{B}^k)' = (2k-n) \oint \mathcal{B}^k \phi  + \oint T_{k-1} (A)^n_j
   \phi^j.$
  Hence, by (\ref{e:S3}) we finally arrive at
  \begin{eqnarray*}
   & &\left(\int \gs_k(A)dV + \oint \mathcal{B}^k d\Sigma - \Lambda \int dV
   \right)'
   =  (2k-n) \left( \int \gs_k \phi + \oint \mathcal{B}^k
   \phi \right) - \Lambda \int n\phi
  \end{eqnarray*} for constant $\Lambda.$
 \epf

 \bpf[Proof of Corollary~\ref{c:variation}]
  Let $g_t= e^{-2u_t} g$ be a conformal
 variation of $g$ such that $u_0= 0$ and $u'_t|_0= \phi.$
 Since $\mathcal{L}(g_t) = e^{(2k-1)u_t} \mathcal{L}(g),$
 we have $\mathcal{L}'= (2k-1)\phi \mathcal{L}.$
 Therefore,
 $\left(\oint \mathcal{L} d\Sigma\right)' = (2k-n) \oint \mathcal{L} d\Sigma.$
 Combining the above formula with the results of
 Theorem~\ref{t:variation} gives
$
   \left(\int \gs_k(A)dV + \oint (\mathcal{B}^k+ \mathcal{L})d\Sigma - \Lambda \int dV
   \right)'
   =  (2k-n) ( \int \gs_k \phi + \oint (\mathcal{B}^k +
   \mathcal{L})
   \phi ) - \Lambda \int n\phi.
  $
  \epf
 \section{Conformal Invariants $\mathcal{Y}_k$} \label{s:conf-inv}

 In this section, we first show that  $\mathcal{F}_{\frac{n}{2}}$ is a conformal invariant and then  we prove Theorem~\ref{t:lcf-inv}.
 Let $\mathcal{L}_4 (g) = -2 \gs_1(A^T) h -2 (n-3) A_{\ga \gb} L^{\ga \gb}
  + 2 R^{\gc}_{\ga \gc \gb} L^{\ga \gb},$ which   satisfies $\mathcal{L}_4 (\hat g)= e^{3u} \mathcal{L}_4 (g)$; see \cite{BG94}. 
 \begin{proposition}\label{p:gb}
  Let $(M, g)$ be a compact manifold of dimension $n \geq 3$ with boundary.
   \\
  (a) When $n = 4$, then $\mathcal{B}^2 = \frac{1}{2}\mathcal{B}+ \frac{1}{4} \mathcal{L}_4.$
    Therefore, $\mathcal{F}_2= 2 \pi^2 \chi (M, \de M)
    -\frac{1}{16} \int |\mathcal{W}|^2+ \frac{1}{4}\oint \mathcal{L}_4$ is a conformal invariant.\\
  (b) Suppose $M$ is  locally conformally flat.
  When $n = 2k,$ then $\mathcal{F}_{\frac{n}{2}} =
  \frac{(2 \pi)^{\frac{n}{2}}}{(\frac{n}{2})!} \chi (M, \de M).$
 \end{proposition}

 \bpf[Proof of Proposition~\ref{p:gb}]  (a) By Lemma~\ref{l:variation} (a), we have
    $
   \mathcal{B}^2 =\gs_{2,1} (A^T, L)+ \gs_{3,0} (A^T, L)
     =  \frac{1}{2} \gs_1 (A^T) h - \frac{1}{2}
   L^{\ga}_{\gb} A^{\gb}_{\ga} + \frac{1}{3} tr L^3 + \frac{1}{6}
   h^3 -  \frac{1}{2}h |L|^2,
  $
   which is equal to $\frac{1}{2}\mathcal{B}+ \frac{1}{4} \mathcal{L}_4$
   by direct computations.
  Since $\mathcal{W}$ and $\mathcal{L}_4$ are local conformal
  invariants, $\mathcal{F}_2$ is then a conformal invariant.

  (b) Recall the  Gauss-Bonnet formulas
   $(4 \pi)^{\frac{n}{2}}\chi(M, \de M) = \int E_n dV + \oint \sum_i Q_{i,n} d\Sigma,$
  where $E_n = (2^{\frac{n}{2}} (\frac{n}{2})!)^{-1} \sum
  \left(  \begin{array}{l}
    i_1 \cdots i_n \\
    j_1 \cdots j_n
    \end{array} \right)
  R_{i_1 i_2}^{\quad j_1 j_2}\cdots R_{i_{n-1} i_n}^{\qquad j_{n-1} j_n}$ and\\
   $Q_{i,n} =  \frac{2^{\frac{n}{2}-2i}}{i! (n-1-2i)!!} \sum \left(  \begin{array}{l}
    \ga_1 \cdots \ga_{n-1}\\
    \gb_1 \cdots \gb_{n-1}
    \end{array} \right)
   R_{\ga_1 \ga_2}^{\quad\;\; \gb_1 \gb_2}\cdots R_{\ga_{2i-1} \ga_{2i}}^{\qquad\quad \gb_{2i-1} \gb_{2i}}
   L_{\ga_{2i+1}}^{ \gb_{2i+1}} \cdots L_{\ga_{n-1}}^{ \gb_{n-1}}.$
   When the manifold is locally conformally flat, by the curvature
   decomposition $R_{ijkl}= A_{ik} g_{jl} + A_{jl}g_{ik} -A_{il}g_{jk}- A_{jk} g_{il}.$
   It has been shown in \cite{Via00} that $E_n = 2^{\frac{n}{2}} (\frac{n}{2})! \gs_{\frac{n}{2}} (A).$
   We only need to compute $Q_{i, n}.$
   \begin{eqnarray*}
    Q_{i,n} &=&  \frac{2^{\frac{n}{2}-2i}}{i! (n-1-2i)!!} \sum \left(  \begin{array}{l}
    \ga_1 \cdots \ga_{n-1}\\
    \gb_1 \cdots \gb_{n-1}
    \end{array} \right)
   2^i(A_{\ga_1}^{\gb_1} g_{\ga_1}^{\gb_1}+ A_{\ga_2}^{\gb_2}
   g_{\ga_1}^{\gb_1})   \cdots \\
   & &(A_{\ga_{2i-1}}^{\gb_{2i-1}} g_{\ga_{2i}}^{\gb_{2i}}+ A_{\ga_{2i}}^{\gb_{2i}}
   g_{\ga_{2i-1}}^{\gb_{2i-1}})  L_{\ga_{2i+1}}^{\gb_{2i+1}} \cdots
   L_{\ga_{n-1}}^{\gb_{n-1}}\\
   &=& \frac{2^{\frac{n}{2}-2i}}{i! (n-1-2i)!!} \sum \left(  \begin{array}{l}
    \ga_1 \cdots \ga_i \, \ga_{2i+1} \cdots \ga_{n-1}\\
    \gb_1 \cdots \gb_i \, \gb_{2i+1}\cdots \gb_{n-1}
    \end{array} \right)
   i! 2^{2i} A_{\ga_1}^{\gb_1}\cdots A_{\ga_i}^{\gb_i} L_{\ga_{2i+1}}^{\gb_{2i+1}} \cdots
   L_{\ga_{n-1}}^{\gb_{n-1}}\\
   &=& \frac{2^{\frac{n}{2}}(n-1-i)!}{(n-1-2i)!!} \gs_{n-1-i,
   i}(A^T, L)= 2^{\frac{n}{2}} (\frac{n}{2})! C_1 (n, \frac{n}{2}, i).
   \end{eqnarray*} \epf


 \bpf[Proof of Theorem~\ref{t:lcf-inv}]
  We will show that there exists a conformal metric $\hat g$ such
  that $A_{\hat g} \in \Gamma^+_k$ and the boundary is totally
  geodesic. Then by the result in \cite{Chen05a}, we can find a
  conformal metric $\tilde{g}$ such that $\gs_k (A_{\tilde{g}}) = 1$
 and the boundary is totally geodesic.

  Let the background
  metric $g$ be a Yamabe metric such that $R= constant > 0$ and
  the boundary is totally geodesic. 
  We  prove inductively that we can
  find $\hat g$ such that $A_{\hat g} \in \Gamma^+_m$ for $m \leq k.$
  Suppose $g$ satisfies $A_g \in \Gamma^+_{m-1}$ and the boundary
  is totally geodesic.
  Define
  $A^t_{m-1} =  A+ \frac{1-t}{2} \gs_{m-1}^{\frac{1}{m-1}}(A)g.$
  Under the conformal change $\hat g = e^{-2u}g$, the tensor $\hat A^t_{m-1}$ satisfies
  $ \hat A^t_{m-1}
  = \hat A+ \frac{1-t}{2}
  \gs_{m-1}^{\frac{1}{m-1}} (g^{-1}\hat A)g,
 $ where $\hat A = \hess u + du \otimes du  - \frac{1}{2} |\gra u|^2 g + A.$
 Since $\gs_{m-1}(A)$ is positive, we  choose a large number $\Theta$
  such that $ A^{-\Theta}_{m-1}$ is positive
 definite. Let $f(x) = \gs_m^{\frac{1}{m}} (A^{-\Theta}_{m-1})> 0$.
Consider the following path of equations for $- \Theta \leq t \leq
1:$
 \begin{equation} \label {e:lcf-pos}\left\{  \begin{array}{ll}
 \gs_m^{\frac{1}{m}} (g^{-1} \hat A^t_{m-1})= f(x) e^{2u} & in \, M  \\
  \frac{\de u}{\de n} = 0 & on \,\de M,
  \end{array}\right .
  \end{equation}
 where $\hat A \in \Gamma^t_m =\{\gl : \gl \in \Gamma^+_{m-1}, \gl +
  \frac{1-t}{2}\gs_{m-1}^{\frac{1}{m-1}}(\gl)e \in \Gamma^+_m \}.$
  Note that if $\gl +  \frac{1-t}{2}\gs_{m-1}^{\frac{1}{m-1}}(\gl)e \in \Gamma^+_m,$
  then we must have $\gl \in \Gamma^+_{m-1}$ along the path
  because $\gs_{m-1} (\gl +
  \frac{1-t}{2}\gs_{m-1}^{\frac{1}{m-1}}(\gl)e)$ can not be zero.

 Let $\mathcal{S} = \{ t\in [-\Theta, 1]: \exists$ a solution $u \in C^{2, \ga}(M)$
 to (\ref{e:lcf-pos}) with $ \hat A \in \Gamma^t_m\}.$ At $t= -\Theta$, we have
 $u \equiv 0$ is a solution and $A^{-\Theta}_{m-1} \in \Gamma^+_m$.  Consider the
 linearized operator $\mathcal{P}^t:$
 \begin{lemma}
  The linearized operator $\mathcal{P}^t: C^{2, \ga}(M) \cap \{\frac{\de u}{\de n}|_{\de M} = 0\}
 \rightarrow C^{\ga} (M)$ is invertible.
 \end{lemma}
 \bpf
  Let $F^t= \gs_m (g^{-1} \hat A^t_{m-1}) - f^m e^{2mu}$  and $u_s$
  be a variation of $u$ such that $u'= \phi$ at $s=0.$ Then
  \begin{eqnarray*}
  \mathcal{P}^t &=& (F^t)'|_{s=0} = T_{m-1} (g^{-1} \hat A^t_{m-1})^{ij}  (g^{-1} \hat A^t_{m-1})_{ij}'
   - 2m f^m e^{2mu} \phi\\
    &=& [T_{m-1} (g^{-1} \hat A^t_{m-1})^{ij} +
    \frac{1-t}{2} \gs_{m-1}^{-\frac{m-2}{m-1}} (g^{-1} \hat
    A) tr_g T_{m-1} (g^{-1} \hat A^t_{m-1}) T_{m-2}(g^{-1}\hat A)^{ij}] \phi_{ij}  \\
    &+& \text{1st derivatives in}\, \phi   - 2m f^m e^{2mu} \phi.
  \end{eqnarray*}
   Since the terms in the parenthesis are positive,   the linearized operator is invertible.
 \epf
 The above lemma and the implicit function theorem imply that $\mathcal{S}$ is open.
 To complete the proof, it remains to establish a priori estimates for solutions to
 (\ref{e:lcf-pos}).

 (1) $C^0$ estimates.

   Since $\frac{\de u}{\de n} = 0,$ at the maximal point $x_0$ of $u,$  we have $|\gra u| = 0$ and
    $\hess u (x_0)$ is negative
   semi-definite, no matter $x_0$ being in the interior or at the boundary. 
   Hence,
   $
    f(x_0) e^{2u(x_0)} = \gs_m^{\frac{1}{m}}(g^{-1} \hat A^t_{m-1})
     \leq \gs_m^{\frac{1}{m}} ( \frac{1-t}{2}
    \gs_{m-1}^{\frac{1}{m-1}} (A)g + A) \leq C,
   $
   where in the  inequality we use $t \leq 1.$
  Therefore, $u$ is  upper bounded.

   Now by \cite{Chen05} and Theorem~\ref{t:bdy} (a), we have $|\gra u|<
   C.$ Thus, $\sup_M u \leq \inf_M u + C.$
    Integrating the equation,
  $$\begin{array} {ll}
   \int f^m e^{4 m u} dV_{\hat g} &= \int e^{2mu} \gs_m ( g^{-1} \hat A^t_{m-1}) dV_{\hat g}
   =\int \gs_m (\hat g^{-1} \hat A^t_{m-1}) dV_{\hat g}\\
    &= \int \sum_{i=0}^m \binom{n-i}{m-i} (\frac{1-t}{2})^{m-i} \gs_i (\hat g^{-1} \hat A)
    \gs_{m-1}^{\frac{m-i}{m-1}} (\hat g^{-1} \hat A) dV_{\hat g} \geq \int \gs_m
    (\hat g^{-1} \hat A) dV_{\hat g},
   \end{array}$$
   where we drop the terms for $i= 0, \cdots, m-1,$ which are
   nonnegative. Since the boundary is totally geodesic, we have $\mathcal{B}^k= 0.$
   Therefore,
   \begin{eqnarray*}
   0 < \mathcal{Y}_m \leq \frac{\int f^m e^{4 m u} dV_{\hat g}}{(\int dV_{\hat g})^{\frac{n-2m}{n}}}
    \leq C (\sup e^{4m u}) (\int dV_{\hat g})^{\frac{2m}{n}} \leq
    C e^{4m \sup u} e^{-2m \inf u}.
   \end{eqnarray*} Since $\sup_M u \leq \inf_M u + C,$ we then
   have $0 < \mathcal{Y}_m \leq C e^{2m \sup u + C}.$

 (2) $C^{\infty}$ estimates

  By \cite{Chen05} and Theorem~\ref{t:bdy} (a), we get interior and boundary $C^2$
  estimates, respectively. Higher order regularity follows the
  same way as in (3) in the proof of
  Proposition~\ref{p:pos}.
 \epf

 \section{Proofs of Theorem~\ref{t:bdy} and Corollary~\ref{c:lcfbdy}} \label{s:bdy}

  \bpf[Proof of Theorem~\ref{t:bdy}]
 Let $W = \hess u + du \otimes du -\frac{1}{2}|\gra u|^2 g + S(x).$
 The condition $\Gamma^+_1 \subset \Gamma$
 gives $0 < tr_g\, \hat{A} = \Delta u -\frac{(n-2)}{2}|\gra u|^2+ tr_g S(x).$
 Thus, $\Delta u $ has a lower bound and
 \beq \label{i:gra}
 |\gra u|^2 < C (\Delta u + 1).
 \eeq
 We first prove a lemma which will be used later to control the
 boundary behavior of $u.$
 \begin{lemma}\label{l:S-bdy} Let $W$ be defined as above.
   Under the same conditions as in Theorem~\ref{t:bdy}, we have\\
  (a) $W_{n \ga} =0$ on $\de M$ and hence $F_{\ga n} =0$ on $\de
  M;$\\
   (b) $W_{\ga \gb, n} - 2 \mu W_{\ga \gb} \leq - \hat \mu e^{-u} (W_{\ga \gb}+ W_{nn} g_{\ga \gb}).$
 \end{lemma}
  \bpf
  (a) By (\ref{e:nga}) and (T0),
     $$W_{\ga n} = u_{\ga n}+ u_n u_{\ga} + S_{\ga n} =
     -\mu_{\ga} + \mu u_{\ga} - \hat \mu u_{\ga} e^{-u} + (- \mu + \hat \mu e^{-u}) u_{\ga}
     + S_{\ga n} =0.$$
    To prove $F_{\ga n} = 0,$  since $F$ is a function of $\gs_i$, we only need to
 show that $\frac{\de \gs_i (W)}{\de W_{\ga n}} = (T_{i-1})_{\ga n} = 0$
for all $i.$ For $i= 1,$ by definition $(T_1)_{\ga n} = \gs_1(W)
g_{\ga n} - W_{\ga n}=0.$ For general $i,$ notice the recursive
relation $(T_i)_{\ga n} = \gs_i (W) g_{\ga n} - (T_{i-1})_{\ga j}
W_{j n}.$ Applying the induction hypothesis gives $(T_i)_{\ga n} =
- (T_{i-1})_{\ga \gb} W_{\gb n} = 0.$

  (b) By (\ref{e:nga}) and (\ref{e:ngagb}),
    \begin{eqnarray*}
     W_{\ga \gb, n} &=& u_{\ga \gb n} + u_{\ga} u_{\gb n} + u_{\ga
     n}u_{\gb} - u_l u_{ln} g_{\ga \gb} + S_{\ga \gb}\\
    &=& (2 \mu - \hat \mu e^{-u}) (u_{\ga \gb} +
     u_{\ga} u_{\gb}) - \mu_{\tilde{\ga} \tilde{\gb}}+
      R_{n \gb \ga n} (-\mu + \hat \mu e^{-u}) - \mu u_n^2 g_{\ga
      \gb}\\
      & &- (\mu - \hat \mu e^{-u}) u_{\gc}^2 g_{\ga \gb} - \hat \mu
      e^{-u} u_{nn} g_{\ga \gb} + S_{\ga \gb, n}.
    \end{eqnarray*}
    Therefore,
    \begin{eqnarray*}
     W_{\ga \gb, n} 
     &=& (2 \mu - \hat \mu e^{-u}) W_{\ga \gb} -\hat \mu e^{-u}
     W_{n n}- (2 \mu - \hat \mu e^{-u})S_{\ga \gb}
      - \mu_{\tilde{\ga} \tilde{\gb}}+
      R_{n \gb \ga n} (-\mu + \hat \mu e^{-u})\\
    & &  + \hat \mu e^{-u} S_{nn} g_{\ga \gb} + S_{\ga \gb, n}.
    \end{eqnarray*}
   Now by (T1) and (T2), we    arrive at
   \begin{eqnarray*}
     W_{\ga \gb, n} &\leq& (2 \mu - \hat \mu e^{-u}) W_{\ga \gb} -\hat \mu e^{-u}
     W_{n n}+ \hat \mu e^{-u}S_{\ga \gb}  - R_{\gb n \ga n} \hat \mu e^{-u}
    + \hat \mu e^{-u} S_{nn} g_{\ga \gb}\\
    &\leq& (2 \mu - \hat \mu e^{-u}) W_{\ga \gb} -\hat \mu e^{-u}
     W_{n n},
   \end{eqnarray*}
   where the last inequality is by nonnegativity of $\hat \mu$.
  \epf
  We continue the proof of Theorem~\ref{t:bdy}.

 (1) We show that on the boundary $u_{nnn}$ can be  controlled from below
  by $\Delta u.$ More specifically, we have $u_{nnn} \geq -L \Delta u + 3 \mu
  u_{nn} - C$ for some number $L$ independent of points on the boundary.

 At a boundary point, differentiating the equation on both sides in the normal
 direction, we get
 $$ (f(x, u))_n = F^{\ga \gb} W_{\ga \gb,n} +
 F^{nn} W_{nn,n},$$ where we have used $F^{\alpha n} = 0$ by
 Lemma~\ref{l:S-bdy}.

  For case (a), by Lemma~\ref{l:S-bdy} again,
 $W_{\ga \gb, n} - 2 \mu W_{\ga \gb} \leq 0$. Thus,
  \begin{align}
   (f(x, u))_n &\leq  2 \mu F^{\ga \gb} W_{\ga \gb}
    + F^{nn} W_{nn,n} \notag \\
    &= 2 \mu F + F^{nn} (W_{nn,n} - 2 \mu W_{nn})
   = 2 \mu f(x,u) +  F^{nn} (W_{nn,n} - 2 \mu W_{nn}), \label{i:a}
  \end{align} where the first equality holds by Lemma~\ref{l:sym}
  (a).
 By (\ref{e:nga}) and the boundary condition,
  \begin{eqnarray*}
  W_{nn,n}- 2 \mu W_{nn} &=& u_{nnn} +2 u_n u_{nn} -
  u_l u_{ln} + S_{nn,n}-2\mu (u_{nn} +  u_n^2 - \frac{1}{2} |\gra u|^2 + S_{nn})\\
  &=&  u_{nnn} - 3 \mu u_{nn} + u_{\ga} \mu_{\ga}+ S_{nn,n} -\mu^3
  - 2 \mu S_{nn}.
  \end{eqnarray*}
 Returning to (\ref{i:a}), we use the conditions $|\gra_x f|\leq \Lambda f$ and $|f_z|\leq \Lambda f$
 to get
 $$  -Cf  \leq f_{x_n} + f_z u_n - 2 \mu f \leq  F^{nn} ( W_{nn,n}- 2 \mu W_{nn})
     \leq F^{nn} (u_{nnn} - 3 \mu u_{nn} + u_{\ga} \mu_{\ga}+ C).
 $$

  On the other hand, by condition (S3) we have
 $ F^{nn} \geq \epsilon \frac{F}{\gs_1} \geq  \epsilon\frac{f(x, u)}{\Delta u+ C} .$
 Hence, there is a positive number $L$ such that
 \beq \label{i:third}
 u_{nnn} \geq -L \Delta u + 3 \mu u_{nn} - u_{\ga} \mu_{\ga}- C
 \eeq is true for every point on the boundary, where $L$ and $C$
 depend on $n, \epsilon, \mu, c_{\sup}$ and $ \Lambda.$


 For case (b), by Lemma~\ref{l:S-bdy} (b) we get
  \begin{align}
   (f (x, u))_n &\leq \sum_{\ga, \gb} F^{\ga \gb} ( 2\mu W_{\ga
   \gb}- \hat \mu e^{-u} (W_{\ga \gb} + W_{nn} g_{\ga
   \gb}))  + F^{nn} W_{nn,n} \notag \\
    &= (2 \mu - \hat \mu e^{-u}) f(x,u) - \hat \mu e^{-u} \sum_{\ga}F^{\ga \ga} W_{nn}
     + F^{nn} (W_{nn,n} - (2 \mu - \hat \mu
     e^{-u})W_{nn}),\notag
  \end{align} where the  equality holds by Lemma~\ref{l:sym}
  (a). Using the conditions $|\gra_x f|\leq \Lambda f$ and $|f_z|\leq \Lambda f,$
   the above formula becomes
   $$
    -C f \leq f_{x_n} + f_z u_n -  (2 \mu - \hat \mu e^{-u}) f
   \leq- \hat \mu e^{-u} \sum_{\ga}F^{\ga \ga} W_{nn}
     + F^{nn} (W_{nn,n} - (2 \mu - \hat \mu
     e^{-u})W_{nn}),
  $$ where $C$ depends on $\inf u.$
  Since $\hat \mu$ is positive, if $W_{nn} \geq 0,$
  then
   $ -C f\leq F^{nn} (W_{nn,n} - (2 \mu - \hat \mu
   e^{-u})W_{nn}).$    
  On the other hand, if $W_{nn} < 0,$ by condition (A) we
  have $ -C f\leq F^{nn} (W_{nn,n} - (2 \mu + \rho \, \hat
  \mu e^{-u})W_{nn}),$ where we drop the term $ F^{nn} \hat \mu
     e^{-u} W_{nn}$ since it is negative.
  Hence, in both cases we obtain
   \beq \label{i:a'}
   -C f\leq F^{nn} (W_{nn,n} - 2 \mu W_{nn} +
   C |W_{nn}|).\eeq
  Now  by (\ref{e:nga}) and (\ref{e:ngagb}) and combined with  a basic fact
  that if $\Gamma^+_2 \subset \Gamma,$ then $|u_{ij}| \leq C \Delta u,$
   we get
 $$  W_{nn,n}- 2 \mu W_{nn} + C |W_{nn}|
     \leq u_{nnn} +(- 3 \mu + \hat \mu e^{-u}) u_{nn} + C \Delta u + C.
 $$
 Returning to (\ref{i:a'}), note that by condition (S3) we have
 $ F^{nn} \geq \epsilon \frac{F}{\gs_1} \geq  \epsilon\frac{f(x, u)}{\Delta u+ C}.$
 Hence, there is a positive number $L$ such that
 \beq \label{i:third'}
 u_{nnn} \geq -L \Delta u + (3 \mu - \hat \mu e^{-u}) u_{nn} - C
 \eeq is true for every point on the boundary, where $L$ and $C$
 depends on $n, \epsilon, \rho, \mu, \hat \mu, \inf u, c_{\sup}$ and $\Lambda.$
\vskip 1em
 (2) We will show that $\Delta u$ is bounded. The follow proof is for
 both cases (a) and (b), while the number $C$ is understood as a
 constant depending on $ n, r, \epsilon, \mu, c_{\sup}$ and
 $\Lambda$ for case (a), and $ n, r, \epsilon, \rho, \mu, \hat \mu, \inf u,
 c_{\sup}$ and $\Lambda$ for case (b), respectively.

Define $\mu$ on the half ball in  Fermi coordinates by $\mu
 (x',x_n) =\mu (x'),$ where $x'= (x_1, \cdots, x_{n-1}).$
Let $H = \eta ( \Delta u + |\gra u|^2 + n\mu \, u_n) e^{a\,x_n} =
\eta K e^{a\,x_n}$ where $a$ is some number chosen later. Denote
$r^2 :\equiv \sum_i x^2_i.$ Let $\eta(r)$ be a cutoff function
such that $0 \leq \eta \leq 1$, $\eta = 1$ in
$\overline{B}^+_{\frac{r}{2}}$ and $\eta = 0$ outside
$\overline{B}^+_r,$ and also $|\gra \eta|< C
\frac{{\eta}^{\frac{1}{2}}}{r}$ and $|\hess \eta|< \frac{C}{r^2}.$
By (\ref{i:gra}), $\Delta u$ is lower bounded. Without loss of
generality, we may assume $r= 1$ and
 $$K= \Delta u + |\gra u|^2 + n\mu \, u_n \gg 1.$$

 At a boundary point,  since $\eta = \eta (r),$ we have $\eta_n = 0.$
 Differentiating $H$ in the normal direction produces
 \begin{eqnarray*}
  H_n&=&\eta (K_n + a K) e^{a x_n}  
     = \eta ( u_{nnn} + u_{\ga \ga n} + (2 u_n + n \mu)u_{nn} +2 u_{\ga} u_{\ga n} + a K)
     e^{a x_n}.
  \end{eqnarray*}
 Using (\ref{e:nga}) and (\ref{e:ngagb}) gives
 \begin{eqnarray*}
  H_n &\geq& \eta ( u_{nnn} - \tilde{\Delta} \mu + (2 \mu- \hat \mu e^{-u}) u_{\ga \ga}
  +(- \mu+2 \hat \mu e^{-u}) u_{nn}+ (2 \mu -\hat \mu e^{-u})
  u_{\ga} u_{\ga}\\
  & &- (n-1) \mu_{\ga} u_{\ga}
  - \mu (n-1) (- \mu + \hat \mu e^{-u})^2- R_{nn} (- \mu + \hat \mu e^{-u}) + a K -C) e^{a x_n}\\
 &\geq& \eta (  u_{nnn}  - (n-1) \mu_{\ga} u_{\ga}+ (2 \mu- \mu e^{-u})  K + (- 3 \mu+ \mu e^{-u}) u_{nn}  + a
 K- C)e^{a x_n}.
\end{eqnarray*}

By (\ref{i:gra}) and the inequalities (\ref{i:third}) and
(\ref{i:third'}) for
 cases (a) and (b), respectively, we obtain
  \begin{eqnarray*}
 H_n &\geq& \eta ( - L \Delta u + (2 \mu - \hat \mu e^{-u}) K  - (n-1)\mu_{\ga} u_{\ga}
 -C  + a K)e^{a
 x_n}> 0
\end{eqnarray*}
for $a > L - 2\mu + \hat \mu \sup e^{-u} +1.$ Thus, $H$ increases
toward the interior and the maximum of $H$ must happen at some
point $x_0$ in the interior.


Now we know the maximal point $x_0$ is in the interior. Thus, at
$x_0$ we have
\begin{equation}\label{e:star'}
H_i = \eta_i (K e^{a x_n}) + \eta e^{a x_n}(K_i + a K \gd_{in})=0,
\end{equation}and
 $$ H_{ij} = \eta_{ij} ( K e^{a x_n}) + \eta_i ( K e^{a x_n})_j +
  \eta_j ( K e^{a x_n})_i + \eta  (K e^{a x_n})_{ij},$$
is negative semi-definite. Using (\ref{e:star'}), the above
formula becomes
$$ H_{ij} = (\eta_{ij}- 2 \eta^{-1} \eta_i \eta_j) K e^{a x_n} +
    \eta e^{a x_n}(K_{ij} + a K_i \gd_{jn} + a K_j \gd_{in} + a^2 K
    \gd_{in} \gd_{jn}).
 $$

Using the positivity of $F^{ij},$ and (\ref{e:star'}) to replace
$K_i$ and $K_j$, we get
 $$
 0 \geq F^{ij} H_{ij} e^{-a x_n} = F^{ij} ((\eta_{ij}- 2 \eta^{-1} \eta_i \eta_j) K +
    \eta (K_{ij} - a \frac{\eta_i}{\eta} K \gd_{jn} - a \frac{\eta_j}{\eta} K
    \gd_{in}- a^2 K \gd_{in} \gd_{jn})).
 $$
 Therefore,
  \beq \label{i:b}
     0 \geq   \eta F^{ij} K_{ij} -C \sum_i F^{ii} K,
 \eeq
 where we use conditions on $\eta.$

 By direct computations, we have
 $$
  F^{ij} K_{ij} = F^{ij}(u_{llij} +  2 u_{li}u_{lj} + 2 u_l u_{lij}
  + n \mu_{ij} u_n + n \mu_i u_{nj} + n \mu_j u_{ni}+ n \mu u_{nij}).$$
 Changing the order of the covariant differentiations and using
(\ref{i:gra}) give
  \begin{eqnarray*}
   F^{ij} K_{ij} &\geq& F^{ij}u_{ijll} + F^{ij} (2 u_{li}u_{lj} + 2 u_l
   u_{ijl}+ n \mu u_{ijn}) - C \sum_i F^{ii} (1 + |\hess u|)\\
    &=& I + II - C \sum_i F^{ii} (1 + |\hess u|).
  \end{eqnarray*}
 For I, notice that
 $$W_{ij,ll} = u_{ijll} + 2 u_{il}u_{jl} + u_i u_{jll} + u_j u_{ill} -
   (u_k u_{kll} + u^2_{kl})g_{ij} + S_{ij,ll}.$$
 Then
 $$
  I = F^{ij} (W_{ij,ll}- 2 u_{li}u_{lj}- 2 u_{ill} u_j +
  (u_{lk} ^2 + u_k u_{kll}) g_{ij} - S_{ij,ll}),$$
 where $F^{ij}(u_i u_{jll})= F^{ij}(u_j u_{ill})$ because $F^{ij}$ is symmetric.
 Changing the order of differentiations again yields
  $$
  I \geq F^{ij} W_{ij,ll}+ F^{ij}( - 2 u_{li}u_{lj}- 2 u_{lli} u_j +
  (u_{lk} ^2 + u_k u_{llk}) g_{ij})- C \sum_i F^{ii} (1 + |\hess
  u|).$$
 Now replace $u_{lli}$ and $u_{llk}$ by (\ref{e:star'})
 to get
  \begin{eqnarray*}
   I &\geq& F^{ij}W_{ij,ll} + F^{ij}( - 2 u_{li}u_{lj} -2 u_j (-2 u_l u_{li}
   -n \mu u_{ni}- n \mu_i u_n - \frac{\eta_i}{\eta} K - a K \gd_{in})\\
    & &  + (|\hess u|^2 + u_k (-2 u_l u_{lk}- n\mu u_{nk}- n \mu_k u_n - \frac{\eta_k}{\eta} K - a K \gd_{kn}))
   g_{ij})\\
   & & - C \sum_i F^{ii} (1 + |\hess
  u|).
  \end{eqnarray*}
 By (\ref{i:gra}) and the conditions on $\eta,$ we have
   \begin{eqnarray*}
   I &\geq& F^{ij}W_{ij,ll} + F^{ij}( - 2 u_{li}u_{lj} + 4 u_j u_l u_{li}
     +(|\hess u|^2 - 2 u_k u_l u_{lk}) g_{ij})\\
    & &- C \sum_i F^{ii}\eta^{- \frac{1}{2}}( 1 + |\hess u|^{\frac{3}{2}}).
  \end{eqnarray*}

 For II, we use the formula
 $$W_{ij,l} = u_{ijl} + u_i u_{jl}+ u_j u_{il}- u_k u_{kl} g_{ij} + S_{ij,l}$$
 to obtain
  \begin{eqnarray*}
  II &=& F^{ij} (2 u_{li}u_{lj}+ 2 u_l (W_{ij,l} - 2 u_i
   u_{jl} + u_k u_{kl} g_{ij}- S_{ij,l})\\
    & &+ n \mu (W_{ij,n}
   -2 u_i u_{jn} + u_k u_{kn} g_{ij}- S_{ij,n}))\\
   &\geq& F^{ij} (2 u_{li}u_{lj}+ 2 u_l W_{ij,l} - 4 u_i
   u_{jl} u_j + 2 u_k u_{kl} u_l g_{ij}+ n \mu
   W_{ij,n}) \\
   & &- C \sum_i F^{ii} (1+ |\hess u|^{\frac{3}{2}}).
  \end{eqnarray*}

 Combining I and II together, we find that
   \begin{eqnarray*}
    F^{ij} K_{ij}    &\geq& F^{ij}W_{ij,ll} + F^{ij}( - 2 u_{li}u_{lj} + 4 u_j u_l u_{li}
     +(|\hess u|^2 - 2 u_k u_l u_{lk}) g_{ij})\\
    & &+ F^{ij} (2 u_{li}u_{lj}+ 2 u_l W_{ij,l} - 4 u_i
   u_{jl} u_j + 2 u_k u_{kl} u_l g_{ij}+ n \mu
   W_{ij,n})\\
  & &  - C \sum_i F^{ii} \eta^{-\frac{1}{2}}( 1 + |\hess u|^{\frac{3}{2}}).
  \end{eqnarray*}
 Here is the key step of the proof. Three terms from I cancel out
 three terms from II. Thus, after the cancellations we arrive at
    \begin{eqnarray*}
  F^{ij} K_{ij} &\geq& F^{ij}W_{ij,ll} + F^{ij}|\hess u|^2  g_{ij}
    + F^{ij} ( 2 u_l W_{ij,l}+ n \mu W_{ij,n})\\
  & & - C \sum_i F^{ii} \eta^{-\frac{1}{2}}( 1 + |\hess u|^{\frac{3}{2}}).
  \end{eqnarray*}
Now returning to (\ref{i:b}),  applying $\eta$ on both sides
produces
 \begin{eqnarray*}
 0 &\geq& \eta^2 F^{ij}W_{ij,ll} + \eta^2 F^{ij}|\hess u|^2  g_{ij}
    + \eta^2 F^{ij} ( 2 u_l W_{ij,l}+ n \mu W_{ij,n})\\
   & & - C \sum_i F^{ii} ( 1 + \eta^{\frac{3}{2}} |\hess u|^{\frac{3}{2}}).
  \end{eqnarray*}
 By the concavity of $F,$ we have
 $F^{ij} W_{ij,ll} \geq (f(x,u))_{ll}.$ Hence,
  \begin{eqnarray*}
 0 &\geq&  \eta^2 \sum_i F^{ii}|\hess u|^2 + \eta^2 (f(x,u))_{ll}
    + 2 \eta^2 u_l (f (x,u))_l + n \mu \eta^2 (f(x,u))_n\\
  & &  - C \sum_i F^{ii} ( 1 + \eta^{\frac{3}{2}} |\hess u|^{\frac{3}{2}})\\
   &\geq& \sum_i F^{ii} ( \eta^2 |\hess u|^2   - C - C \eta |\hess u|- C \eta^{\frac{3}{2}} |\hess
   u|^{\frac{3}{2}}).
  \end{eqnarray*}
 This gives $(\eta |\hess u|)(x_0) \leq C.$ Hence, for $x \in
\overline{B}^+_{\frac{r}{2}},$ we have that $H = ( \Delta u +
|\gra u|^2 + n \mu \, u_n) e^{a\,x_n}$ is bounded. Thus, $\Delta
u$ is bounded. By (\ref{i:gra}), $|\gra u|$ is also bounded.
\vskip 1em

 (3) To get the Hessian bounds, for case (b) it
 follows immediately by the fact that if $\Gamma^+_2 \subset \Gamma,$
 then $|u_{ij}| \leq C \Delta u.$
 As for case (a), note that from (2) above, we have $\eta \Delta u <
 C$ and $\eta |\gra u|^2 < C.$ Consider the maximum of $\eta (\hess u + du \otimes du
 + \mu u_n g) e^{a x_n}$ over the set $(x, \xi) \in (B_1^+, \mathbb{S}^n).$
 We will show that at the maximum,  $x$ can not belong to the
 boundary. If $\xi$ is in the tangential direction, without loss
 of generality, we can assume $\xi$ is in $e_1$ direction. By formulas (\ref{e:nga})
 and (\ref{e:ngagb}), we obtain
 \begin{eqnarray*}
 & &(\eta(u_{11} + u_1^2 + \mu u_n) e^{a x_n})_n  \\
    &=& \eta e^{a x_n} ( (2 \mu + a) (u_{11} + u_1^2 + \mu u_n) +
    \mu^3- \mu_{\tilde{1} \tilde{1}}- \mu_{\ga}u_{\ga}- R_{n 1 1 n} \mu)\\
    &\geq& \eta e^{a x_n} ( (2 \mu + a) (u_{11} + u_1^2 + \mu u_n) - \mu_{\ga}u_{\ga} -C )> 0
 \end{eqnarray*} for $a > -2\mu + 1.$ If $\xi$ is in the normal
 direction, we first have that $\Delta u \leq n(u_{nn} + \mu^2) \leq n u_{nn} + C.$
 By (\ref{i:third}) and (\ref{i:gra}), we obtain
  \begin{eqnarray*}
  (\eta (u_{nn} + u_n^2 + \mu u_n) e^{a x_n})_n &=& \eta (u_{nnn} - \mu u_{nn} +
   au_{nn}) e^{a x_n}\\
  &\geq& \eta e^{a x_n} (-L \Delta u + 2 \mu u_{nn} + a u_{nn}- C_0 \Delta u- C)\\
  &\geq&  \eta e^{a x_n} (-n (L+C_0) u_{nn} + 2 \mu u_{nn} + a u_{nn}-
  C)> 0
  \end{eqnarray*} for $a> n(L+C_0)- 2\mu + 1.$ Thus, we  conclude that at the
  maximum, $x$ must be in the interior. We then perform
  similar computations as before using the inequality $\eta |\gra u|^2 < C$
  to get the Hessian bounds. We omit the details here.
 \epf

 \bpf[Proof of Corollary~\ref{c:lcfbdy}]
 It has been proved in Section~\ref{s:backgd} that $A_g$ satisfies
 (T0)-(T2). We only need to verify the dependence of $\Lambda$
 and $C_{\sup}$ in Theorem~\ref{t:bdy}.

 Let $\tilde{f}(x, z) = f(x) e^{-2z}$ and $\Lambda = \frac{\|f\|_{C^1}}{\inf f} + 2.$
 Then $$\begin{array} {ll}
       |\gra_x \tilde{f}|\leq |\gra f| e^{-2z} \leq \Lambda (f
      e^{-2z})= \Lambda \tilde{f}
      &\text{and} \quad |\tilde{f}_z|= 2 f e^{-2z}\leq \Lambda
      \tilde{f}.
      \end{array} $$
 For $c_{\sup},$ it is easy to see that $c_{\sup}\leq C \|f\|_{C^2} \sup e^{-2u}
 = C(\|f\|_{C^2}, \inf u).$
 \epf

 \section{Proof of Theorem~\ref{t:bdy1}} \label{s:bdy1}

In this section, we prove Theorem~\ref{t:bdy1}.

 \begin{proof}

 (a) Let $\hat A = \hess u + du \otimes du - \frac{1}{2} |\gra u|^2 g +
 A_g$ and $W = \hat A + S.$
   Recall that $T_1(W) = (tr_g W )g - W$ is
   the first Newton tensor and $F^{ij} = \frac{1}{2 F}
   (T_1)_{ij},$ where $F = \gs^{\frac{1}{2}}_2.$
   Since $F^{ij}$ is positive,  we have
   $$T_1(W)_{nn} =  u_{\ga
 \ga} -\frac{n-3}{2} |\gra u|^2 - u_n^2 + T_1(A)_{nn}+ tr S - S_{nn} >0.$$
 Thus,
 \beq \label{i:grad2}
 |\gra u|^2< C(1 + u_{\ga \ga}).
 \eeq

 We will show that $u_{\ga \ga}$ and hence $|\gra u|^2$ are
 bounded. Define $\mu$ on the half ball in  Fermi coordinates by $\mu
 (x',x_n) =\mu (x'),$ where $x'= (x_1, \cdots, x_{n-1}).$
 Let $G= \eta (u_{\ga \ga}+
 u_{\ga} u_{\ga} + (n-1)\mu u_n) e^{a x_n}= \eta E e^{a x_n},$ where $a$ is some
 number chosen later. Denote $r^2 :\equiv \sum_i
x^2_i.$ Let $\eta(r)$ be a cutoff function as in the proof of Theorem~\ref{t:bdy} (2).
 Without loss of generality, we may assume $r = 1$ and
 $$E= u_{\ga \ga}+ u_{\ga} u_{\ga} + (n-1)\mu u_n \gg 1.$$
 Therefore,  by (\ref{i:grad2}) we get $u_{\ga \ga}\gg 1$.
 Hence, we also have  $u_{\ga} u_{\ga} < E$ on the
 boundary.

 At a boundary point,  since $\eta = \eta (r),$ we have $\eta_n = 0.$
 Differentiating $G$ in the normal direction produces
 \begin{eqnarray*}
  G_n=\eta (E_n + a E) e^{a x_n}  
     = \eta ( u_{\ga \ga n} +  2 u_{\ga} u_{\ga n} + (n-1) \mu u_{nn}+ a E) e^{a
     x_n}.
  \end{eqnarray*}
 Using (\ref{e:nga}) and (\ref{e:ngagb}) gives
 \begin{eqnarray*}
 G_n &=& \eta ( 2 \mu u_{\ga \ga}+ 2 \mu u_{\ga} u_{\ga}-(n-1) \mu^3
 - \tilde{\Delta} \mu -(n-1) \mu_{\ga} u_{\ga} + \mu R_{nn}+ a E) e^{a x_n}\\
 &\geq& \eta ( 2 \mu E - (n-1) \mu_{\ga} u_{\ga}+ a E- C)e^{a x_n}.
\end{eqnarray*}
 By (\ref{i:grad2}), we obtain
  \begin{eqnarray*}
 G_n &\geq& \eta ((2 \mu + a) E  - (n-1) \mu_{\ga} u_{\ga} -C)e^{a
 x_n}> 0
\end{eqnarray*}
for $a >  - 2\mu  +1.$ Hence, the maximum of $G$ must happen  in the
interior.

Now we know the maximal point $x_0$ is in the interior. Thus, at
$x_0$ we have
\begin{equation}\label{e:star}
G_i = \eta_i (E e^{a x_n}) + \eta e^{a x_n}(E_i + a E \gd_{in})=0,
\end{equation}and
$$ G_{ij} = \eta_{ij} ( E e^{a x_n}) + \eta_i ( E e^{a x_n})_j +
  \eta_j ( E e^{a x_n})_i + \eta  (E e^{a x_n})_{ij}$$
is negative semi-definite. Using (\ref{e:star}), the above formula
becomes
 $$ G_{ij} = (\eta_{ij}- 2 \eta^{-1} \eta_i \eta_j) E e^{a x_n} +
    \eta e^{a x_n}(E_{ij} + a E_i \gd_{jn} + a E_j \gd_{in} + a^2
    E \gd_{in} \gd_{jn}).
 $$
 Moreover, direct computations show
 \begin{eqnarray*}
 E_{ij} &=& u_{\ga \ga ij}+ 2 u_{\ga i}u_{\ga j} + 2 u_{\ga}
u_{\ga ij} + (n-1) \mu_{ij} u_n + (n-1) \mu_i u_{nj} \\
 & &+ (n-1)
\mu_j u_{ni}+ (n-1) \mu u_{nij}.
 \end{eqnarray*}
 Using the positivity of $F^{ij},$ and (\ref{e:star}) to
replace $E_i$ and $E_j$, we get
 \begin{align}
 0 \geq F^{ij} G_{ij} e^{-a x_n} &= F^{ij} ((\eta_{ij}- 2 \eta^{-1} \eta_i \eta_j) E +
    \eta (E_{ij} - a \eta^{-1} \eta_i E \gd_{jn} - a \eta^{-1} \eta_j E
    \gd_{in} \notag \\
    & - a^2 E \gd_{in} \gd_{jn}))
      \geq   \eta F^{ij} E_{ij} -C \sum_i F^{ii} E, \label{i:b2}
 \end{align}
 where we use conditions on $\eta$ in the inequality.

 To compute $F^{ij} E_{ij},$  using the formulas for exchanging
 the order of differentiations the first term in $E_{ij}$ becomes
  \begin{eqnarray*}
  F^{ij} u_{\ga \ga ij} &=& F^{ij} (u_{ij \ga \ga}- R_{m \ga i \ga} u_{mj}
   - R_{mij \ga} u_{m \ga} -R_{m \ga j \ga} u_{m i} + R_{mi \ga j} u_{m
   \ga}\\
   & & - R_{m \ga i \ga,j} u_m + R_{m i \ga j,\ga} u_m)\\
   &\geq& F^{ij} u_{ij \ga \ga} - C \sum_i |F^{n i} u_{nn}| - C
   \sum_i F^{ii} (1 + |\gra u| + \sum_{\ga, \gb} |u_{\ga \gb}|+ \sum_{\ga} |u_{n
   \ga}|).
  \end{eqnarray*}
 Therefore,
  \begin{eqnarray*}
  F^{ij} E_{ij} &\geq& F^{ij}(u_{ij \ga \ga}+ 2 u_{\ga i}u_{\ga j} + 2 u_{\ga} u_{ij\ga}
   + (n-1) \mu u_{ijn}) \\
    & &- C \sum_i |F^{in} u_{nn}|
   - C \sum_i F^{ii} (1 + \sum_{\ga, \gb} |u_{\ga \gb}|+ \sum_{\ga} |u_{n
   \ga}|),
  \end{eqnarray*}
 where we use (\ref{i:grad2}).  Denote
 $I =  F^{ij}u_{ij \ga \ga}$ and
 $II = F^{ij}( 2 u_{\ga i}u_{\ga j} + 2 u_{\ga} u_{ij \ga} + (n-1) \mu u_{ijn}).$

  For I, notice that
 $$W_{ij,\ga \ga} = u_{ij \ga \ga} + 2 u_{i \ga}u_{j \ga} + u_i u_{j \ga \ga} +
  u_j u_{i \ga \ga} - (u_k u_{k \ga \ga} + u^2_{k \ga})g_{ij} + A_{ij, \ga \ga}+ S_{ij, \ga \ga}.$$
 Then
 $$
  I \geq F^{ij} (W_{ij, \ga \ga}- 2 u_{\ga i}u_{\ga j}- 2 u_{i \ga \ga} u_j +
  (u_{k \ga} ^2 + u_k u_{k \ga \ga}) g_{ij}) - C \sum_i F^{ii}.$$
  Exchanging the order of differentiations, the above formula
 becomes
  \begin{eqnarray*}
  I \geq F^{ij} (W_{ij, \ga \ga}- 2 u_{\ga i}u_{\ga j}- 2 u_{\ga \ga i} u_j +
  (u_{k \ga} ^2 + u_k u_{\ga \ga k}) g_{ij}) - C \sum_i F^{ii} (1 + |\gra u|^2),
  \end{eqnarray*}
  where we use (\ref{i:grad2}).
 Now using (\ref{e:star}) to replace $u_{\ga \ga i}$ and $u_{\ga \ga k}$
 yields
  \begin{eqnarray*}
   I &\geq& F^{ij} W_{ij,\ga \ga} + F^{ij}( - 2 u_{\ga i}u_{\ga j} -2 u_j (-2 u_{\ga}
   u_{\ga i}  -(n-1) \mu_i u_n -(n-1) \mu u_{ni} \\
   & &- \eta^{-1} \eta_i E - a E \gd_{in})
     + (u_k (-2 u_{\ga} u_{\ga k}- (n-1)\mu_k u_n- (n-1)\mu u_{nk} - \eta^{-1} \eta_k
     E\\
    & &- a E \gd_{kn})
   + u_{k \ga}^2) g_{ij}) - C \sum_i F^{ii} (1 + |\gra u|^2).
  \end{eqnarray*}
  Noting that $E < C (\sum_{\ga} u_{\ga \ga} + 1).$
 By (\ref{i:grad2}) and the conditions on $\eta,$ we arrive at
   \begin{eqnarray*}
   I &\geq& F^{ij} W_{ij,\ga \ga} + F^{ij}( - 2 u_{\ga i}u_{\ga j} + 4 u_j u_{\ga} u_{\ga i}
     +2 (n-1)\mu u_j u_{ni}+ ( - 2 u_k u_{\ga} u_{\ga k} \\
     & &- (n-1) \mu u_k u_{nk} + u_{k \ga}^2) g_{ij})
    - C \sum_i F^{ii}( 1 + |\gra u|^2 + \eta^{- \frac{1}{2}}(\sum_{\ga} u_{\ga \ga})^{\frac{3}{2}}).
  \end{eqnarray*}

  For II, we use the formula
 $$W_{ij,l} = u_{ijl} + u_i u_{jl}+ u_j u_{il}- u_k u_{kl} g_{ij}+ A_{ij,l}+ S_{ij, l}$$
 to obtain
  \begin{eqnarray*}
  II &=& F^{ij} (2 u_{\ga i}u_{\ga j}+ 2 u_{\ga} (W_{ij,\ga} - 2 u_i
   u_{j \ga} + u_k u_{k \ga} g_{ij} - A_{ij,\ga}- S_{ij,\ga})\\
    & &+ (n-1) \mu (W_{ij,n}
   -2 u_j u_{ni} + u_k u_{kn} g_{ij} - A_{ij,n}- S_{ij,n}))\\
   &\geq& F^{ij} (2 u_{\ga i}u_{\ga j}+ 2 u_{\ga} W_{ij,\ga} - 4 u_i
   u_{j \ga} u_{\ga} + 2 u_k u_{k \ga} u_{\ga} g_{ij}+ (n-1) \mu W_{ij,n} \\
   & &- 2 (n-1) \mu u_{ni} u_j + (n-1) \mu u_k u_{kn} g_{ij}) - C \sum_i F^{ii}(1 + |\gra u|).
  \end{eqnarray*}

 Combining I and II together, we find that
   \begin{eqnarray*}
  I + II &\geq& F^{ij}W_{ij,\ga \ga} + F^{ij}( - 2 u_{\ga i}u_{\ga j} + 4 u_j u_{\ga} u_{\ga i}
     +2 (n-1)\mu u_i u_{nj}+ ( - 2 u_k u_{\ga} u_{\ga k} \\
     & &- (n-1) \mu u_k u_{nk} + u_{k \ga}^2) g_{ij})
    + F^{ij} (2 u_{\ga i}u_{\ga j}+ 2 u_{\ga} W_{ij,\ga} - 4 u_i
   u_{j \ga} u_{\ga} \\
   & &+ 2 u_k u_{k \ga} u_{\ga} g_{ij}+ (n-1) \mu W_{ij,n}
    - 2 (n-1) \mu u_{jn} u_i + (n-1) \mu u_k u_{kn} g_{ij})\\
  & & - C \sum_i F^{ii}( 1 + |\gra u|^2 + \eta^{- \frac{1}{2}}(\sum_{\ga} u_{\ga \ga})^{\frac{3}{2}}).
  \end{eqnarray*}
  Five terms from I cancel out
 five terms from II. Thus, after the cancellations 
  \begin{align}
  I + II &\geq F^{ij}W_{ij,\ga \ga} + F^{ij} ( u^2_{k \ga} g_{ij}+
  2 u_{\ga} W_{ij,\ga}+  (n-1) \mu W_{ij,n} ) \notag \\
  &  - C \sum_i F^{ii}( 1 + |\gra u|^2 + \eta^{- \frac{1}{2}}(\sum_{\ga} u_{\ga \ga})^{\frac{3}{2}}).
  \label{i:ii2}
  \end{align}

 Now returning to (\ref{i:b2}),  applying $\eta$ on both sides
produces
 $$
  0 \geq \eta^2 (I + II)
  - C \eta^2 \sum_i |F^{in} u_{nn}|
   - C \eta^2 \sum_i F^{ii} (1 + \sum_{\ga, \gb} |u_{\ga \gb}|+ \sum_{\ga} |u_{n
   \ga}|)-C \eta \sum_i F^{ii}  E.
  $$
 By (\ref{i:ii2}), the above formula becomes
 \begin{eqnarray*}
  0 &\geq& \eta^2 F^{ij}W_{ij,\ga \ga} + \eta^2 F^{ij}u_{k \ga}^2  g_{ij}
    + \eta^2 F^{ij} ( 2 u_{\ga} W_{ij,\ga}+ (n-1) \mu W_{ij,n}) \\
   & & - C \eta^2 \sum_i |F^{in} u_{nn}|
   - C  \sum_i F^{ii} (1 + \eta \sum_{\ga, \gb} |u_{\ga \gb}|+ \eta \sum_{\ga} |u_{n
   \ga}|+ \eta^{ \frac{3}{2}}(\sum_{\ga} u_{\ga \ga})^{\frac{3}{2}}),
  \end{eqnarray*}
 where we have used the fact that $E \leq C (\sum_{\ga} u_{\ga \ga} + 1)$ and
 (\ref{i:grad2}).
   By the concavity of $F$, we have $F^{ij} W_{ij,\ga \ga} \geq (f(x,u))_{\ga \ga}.$ Hence,
  \begin{eqnarray*}
 0 &\geq&  \eta^2 F^{ij}u_{k \ga}^2  g_{ij} + \eta^2 (f(x,u))_{\ga \ga}
  + 2 \eta^2 u_{\ga} (f(x,u))_{\ga} + (n-1) \mu \eta^2 (f(x, u))_n \\
  & & - C \eta^2 \sum_i |F^{in} u_{nn}|
   - C  \sum_i F^{ii} (1 + \eta \sum_{\ga, \gb} |u_{\ga \gb}|+ \eta \sum_{\ga} |u_{n
   \ga}|+ \eta^{ \frac{3}{2}}(\sum_{\ga} u_{\ga
   \ga})^{\frac{3}{2}}).
  \end{eqnarray*} Therefore,
  \begin{align}
   0 &\geq \eta^2 \sum_i F^{ii} u_{k \ga}^2 - C \sum_i F^{ii}(1 + \eta \sum_{\ga, \gb}
    |u_{\ga \gb}|+ \eta \sum_{\ga} |u_{n
   \ga}|+ \eta^{ \frac{3}{2}}(\sum_{\ga} u_{\ga
   \ga})^{\frac{3}{2}})\notag\\
   &  - C \eta^2 \sum_i |F^{in} u_{nn}|, \label{i:a1}
  \end{align}
  where  we use Lemma~\ref{l:sym} (b).

 The term $|F^{in}u_{nn}|$ can be estimated as follows. Note
 that
 $$|F^{in}u_{nn}| \leq |F^{in} W_{nn}| + C \sum_i F^{ii} (1 + |\gra
 u|^2).$$
 Since $W \in \Gamma^+_2$, a basic algebraic fact says that
 $- \frac{n-2}{n} \gs_1 \leq \gl_i \leq \gs_1,$ where $\gl_i$'s are
 the eigenvalues of $W.$ Therefore,
 $|W_{nn}| \leq C \sum_i W_{ii}$. Recall $F^{ij} = \frac{1}{2 F}  T_{ij}.$
  Hence, we have
  $$
  |F^{i n} W_{nn}| \leq C |F^{i n} \sum_j W_{jj}| \leq C | T_1(W)_{n
  i}|\sum_j  F^{jj}.  $$
  Consequently,
  $$
  |F^{\ga n}u_{nn}| \leq C |T_1(W)_{n \ga}|\sum_j  F^{jj}+ C \sum_i F^{ii} (1 + |\gra
 u|^2) \leq C \sum F^{ii} (1 + |\gra u|^2 + \sum_{\gb} |u_{n
 \gb}|), $$ and
  $$
   |F^{n n}u_{nn}| \leq C |T_1(W)_{n n}|\sum_j  F^{jj}+ C \sum_i F^{ii} (1 + |\gra
 u|^2) \leq C \sum F^{ii} (1 + \sum_{\gb} |u_{\ga
 \gb}|).
   $$

 Returning to (\ref{i:a1}),  we obtain
 \begin{eqnarray*}
  0    & \geq&  \sum_i F^{ii} (\eta^2 \sum_{\ga, \gb}u_{\ga \gb}^2 - C (1 + \eta \sum_{\ga, \gb}
    |u_{\ga \gb}|+ \eta^{ \frac{3}{2}}(\sum_{\ga} u_{\ga
   \ga})^{\frac{3}{2}}))
  \end{eqnarray*}
 This gives $(\eta |u_{\ga \gb}|)(x_0) \leq C.$ Thus, for $x \in
\overline{B}^+_{\frac{r}{2}},$ we have that $G = ( u_{\ga \ga} +
u_{\ga} u_{\ga} + (n-1) \mu \, u_n) e^{a\,x_n}$ is bounded. As a
result, $\sum_{\ga} u_{\ga \ga}- u_n^2$ is upper bounded. On the
other hand, since $T_1(W)_{nn}$ is positive,
 $\sum_{\ga} u_{\ga \ga}- u_n^2 > \frac{n-3}{2} |\gra u|^2 - C.$
 Hence, $|\gra u|$ is  bounded. Consequently, $\sum_{\ga} u_{\ga \ga}$
 is also bounded.

(b)
 Let $\hat A^t = \hat A + \frac{1-t}{2} (tr_g \hat A) g =
 \hess u + du \otimes du - \frac{1}{2} |\gra u|^2 g + \frac{1-t}{2}
 (\Delta u - \frac{n-2}{2} |\gra u|^2)g + A^t,$ where $ - \Theta \leq t \leq 1.$
 Let $W= \hat A^t + S.$
 The condition $W \in \Gamma^+_1$ gives
 $$0 < tr_g W = (3-2t) tr_g \hat A + tr_g S = (3-2t) (\Delta u - \frac{n-2}{2} |\gra u|^2 + A_g)+ tr S.$$
 Therefore, we have
 \beq \label{i:grad3}
 |\gra u|^2 < C (\Delta u + 1).
 \eeq
 In the following proof, we adopt the notation $F^{ij} = \frac{\de F(W)}{\de W_{ij}},$
 where $F = \gs_2^{\frac{1}{2}}.$

 (1) We show that on the boundary $u_{nnn}$ can be  controlled from below
  by $\Delta u.$ More specifically, we have $u_{nnn} \geq -L \Delta u
   - C$ for some number $L$ independent of points on the boundary.

  At a boundary point, note that $T_1(W)_{\ga n} = - W_{\ga n} =
  - \hat A_{\ga n}- S_{\ga n} = 0$ by (\ref{e:nga}), Lemma~\ref{l:bdy} (a) and the assumption on $S$.
   Therefore,
   $F^{\ga n}= \frac{T_1(W)_{\ga n}}{2 F} = 0.$
 Differentiating the equation on both sides in the normal direction
 at a boundary point, we get
 \begin{eqnarray*}
  (f(x, u))_n &=&  F^{\ga \gb} W_{\ga \gb,n} + F^{nn}
   W_{nn,n}\\
  &=& F^{\ga \gb} (W_{\ga \gb, n} - 2 \mu W_{\ga \gb})
   + 2 \mu f(x, u) + F^{nn} (W_{nn, n} - 2 \mu W_{nn})\\
  &=& F^{\ga \gb} (\hat A_{\ga \gb, n} - 2 \mu \hat A_{\ga \gb}
  + S_{\ga \gb, n}- 2 \mu S_{\ga \gb})
  + \frac{1-t}{2} \sum_i F^{ii} (g^{jk} \hat A_{jk,n}- 2 \mu g^{jk} \hat
   A_{ik})\\
   & & +F^{nn} (\hat A_{nn, n} - 2 \mu \hat
   A_{nn}+ S_{nn, n} - 2 \mu S_{nn})
   + 2 \mu f(x, u),
 \end{eqnarray*} where in the second equality we use Lemma~\ref{l:sym}
  (a). Using Lemma~\ref{l:bdy1} and the assumption on $S$,
  we have $g^{\ga \gb}(\hat{A}_{\ga \gb, n} -2 \mu \hat{A}_{\ga \gb})=0$ and
  $g^{\ga \gb} (S_{\ga \gb, n} - 2 \mu S_{\ga \gb}) \leq 0$. Therefore,
   \begin{eqnarray*}
  - C &\leq& (f(x, u))_n - 2 \mu f(x, u)
  = - \frac{W_{\ga \gb}}{2 F} (\hat A_{\ga \gb, n} - 2 \mu \hat A_{\ga \gb}
    + S_{\ga \gb, n} - 2 \mu S_{\ga \gb})\\
  & & + \frac{1-t}{2} \sum_i F^{ii} (\hat A_{nn,n}- 2 \mu \hat
   A_{nn})
   + F^{nn} (\hat A_{nn, n} - 2 \mu \hat
   A_{nn} + S_{nn,n}- 2 \mu S_{nn}).
 \end{eqnarray*}
  By (\ref{e:nga}) and (\ref{e:ngagb}), we can compute directly
  that
   $ \hat A_{\ga \gb, n} - 2 \mu \hat A_{\ga \gb}= -2 \mu A_{\ga \gb}
   + \mu_{\tilde{\ga} \tilde{\gb}} -\mu R_{n \gb \ga n} + A_{\ga \gb,
   n}.$
 Hence, $\hat A_{\ga \gb, n} - 2 \mu \hat A_{\ga \gb} + S_{\ga \gb, n} - 2 \mu S_{\ga \gb}$
 is bounded. Thus,
 \beq \label{i:b1}
 - C \leq  \sum_{\ga, \gb} \frac{C}{ F} |W_{\ga \gb}|
    + F^{nn} (\hat A_{nn, n} - 2 \mu \hat
   A_{nn}+ C) + \frac{1-t}{2} \sum_i F^{ii} (\hat A_{nn,n}- 2 \mu \hat
   A_{nn}).
 \eeq

 On the other hand,
 $$
 0< f(x, u)^2 =  T_1(W)^{\ga \gb} W_{\ga \gb} +
 T_1(W)^{nn} W_{nn}
   = - \sum_{\ga, \gb}(W_{\ga \gb})^2 + T_1(W)_{nn}( tr_g W +
   W_{nn}).
 $$ Using the above formula, (\ref{i:b1}) becomes
  $$
   - C  \leq F^{nn}( tr_g W + W_{nn}+ \hat A_{nn, n} - 2 \mu \hat
   A_{nn}+ C) + \frac{1-t}{2} \sum_i F^{ii} (\hat A_{nn,n}- 2 \mu \hat
   A_{nn}).$$
  Hence,
  \beq  \label{i:a2}
    - C \leq F^{nn}(\hat A_{nn, n} + (1- 2 \mu) \hat A_{nn}
    + \frac{7-5t}{2} tr_g \hat A + C) + \frac{1-t}{2} \sum_i F^{ii} (\hat A_{nn,n}- 2 \mu \hat
   A_{nn}).
  \eeq
  Since $W \in \Gamma^+_2,$ we have $|W_{ij}|< C
  tr_g W.$ This gives $|\hat A_{ij}| < C tr_g \hat A + C,$ and
  $|u_{ij}| < C \Delta u + C$ by (\ref{i:grad3}). We also
  get  that at a boundary point,
  $$ \hat A_{nn,n} = u_{nnn} - \mu u_{nn} + \mu_{\ga} u_{\ga}- \mu u_{\ga} u_{\ga}+ A_{nn,n}$$
  by (\ref{e:nga}).
  Hence, returning to (\ref{i:a2}) we obtain
  $$- C \leq  (F^{nn}+ \frac{1-t}{2} \sum_i F^{ii})(\hat A_{nn, n} + C tr_g \hat A + C)
  \leq (F^{nn}+ \frac{1-t}{2} \sum_i F^{ii})(u_{nnn} + C \Delta u + C).$$
  Finally, since $F= \gs_2^{\frac{1}{2}}$ satisfies (S3),
  we have that
 $ F^{ij} \geq C \frac{F}{tr_g W}g^{ij}= C \frac{F}{(3-2t) tr_g \hat A+ tr_g S} g^{ij}
 \geq \frac{C}{\Delta u + C}g^{ij}.$
 Thus, there is a positive number $L$ such that
 \beq \label{i:third1}
 u_{nnn} \geq -L \Delta u  - C
 \eeq  for every point on the boundary, where $L$ and $C$
 depend on $n, \|\mu\|_{C^2},$ $c_{\sup}$ and $c_{\inf}.$

(2) We will show that $\Delta u$ is bounded. Let $H = \eta (
\Delta u + |\gra u|^2) e^{a x_n} = \eta K e^{a x_n},$ where $a$ is
some number chosen later. Let $\eta(r)$ be a cutoff function as in
(a). Without lost of generality, we may assume $r=1$ and $K=
\Delta u + |\gra u|^2 \gg 1.$
 As a consequence, by (\ref{i:grad3}) we get $\Delta u \gg 1$.

 At a boundary point,   differentiating $H$ in the normal direction produces
 \begin{eqnarray*}
  H_n&=& \eta (K_n + a K) e^{a x_n}  
     = \eta ( u_{nnn} + u_{\ga \ga n} + 2 u_n u_{nn} + 2 u_{\ga} u_{\ga n} + a K) e^{a
     x_n}.
  \end{eqnarray*}
 Using (\ref{e:nga}) and (\ref{e:ngagb}) gives
 \begin{eqnarray*}
 H_n &=& \eta ( u_{nnn} + 2 \mu u_{\ga \ga} - (n+1)\mu u_{nn} -(n-1) \mu^3 +
 2 \mu u_{\ga} u_{\ga } - \widetilde{\Delta} \mu + (n-3)
 u_{\ga}\mu_{\ga}\\
 & &+ \mu R_{nn}+ a K) e^{a x_n}\\
 &\geq& \eta (  u_{nnn} + 2 \mu  K   + a
 K -(n+3) \mu u_{nn} +(n-3) u_{\ga} \mu_{\ga}- C)e^{a x_n}.
\end{eqnarray*}
 Note that $|u_{ij}| < C (\Delta u + 1).$
 Then by (\ref{i:grad3}) and  (\ref{i:third1}), we obtain
  \begin{eqnarray*}
 H_n &\geq& \eta ( - (L+ C_0) \Delta u + (2 \mu + a) K  -C)e^{a
 x_n}> 0
\end{eqnarray*}
for $a > L - 2 \mu + C_0 +1.$ Thus, the maximum of $H$ must happen  in
the interior.

 The rest of proof is similar to that of Theorem~\ref{t:bdy};
 to be precise, formula (\ref{e:star'}) and below. Since the proof
 is almost the same, we just sketch here.

 At the maximal point $x_0,$ we have
\begin{equation}\label{e:star''}
H_i = \eta_i (K e^{a x_n}) + \eta e^{a x_n}(K_i + a K \gd_{in})=0,
\end{equation}and
 \begin{eqnarray*}
 H_{ij}   &=& (\eta_{ij}- 2 \eta^{-1} \eta_i \eta_j) K e^{a x_n} +
    \eta e^{a x_n}(K_{ij} + a K_i \gd_{jn} + a K_j \gd_{in} + a^2 K
    \gd_{in} \gd_{jn})
 \end{eqnarray*}
  is negative semi-definite.

Using the positivity of $F^{ij},$ and (\ref{e:star''}) to replace
$K_i$ and $K_j$, we get
 \begin{align}
 0 \geq F^{ij} H_{ij} e^{-a x_n} 
     &\geq  \eta F^{ij} K_{ij} -C \sum_i F^{ii} K. \label{i:b'}
 \end{align}
 By direct computations, we have
  $$
  F^{ij} K_{ij} = F^{ij}(u_{llij} +  2 u_{li}u_{lj} + 2 u_l
  u_{lij})
   \geq F^{ij}u_{ijll} + F^{ij} (2 u_{li}u_{lj} + 2 u_l
   u_{ijl}) - C \sum_i F^{ii} (1 + |\hess u|).
  $$
Denote $I= F^{ij}u_{ijll}$ and $II = F^{ij} (2 u_{li}u_{lj} + 2
u_l u_{ijl}).$ For I, using the formula of
 $W_{ij,ll},$
  $$
  I \geq  F^{ij} W_{ij,ll}+ F^{ij}( - 2 u_{li}u_{lj}- 2 u_{lli} u_j +
  (u_{lk} ^2 + u_k u_{llk}) g_{ij})- C \sum_i F^{ii} (1 + |\hess
  u|).
  $$
 Now replacing $u_{lli}$ and $u_{llk}$ by (\ref{e:star''})
 produces
  \begin{eqnarray*}
   I &\geq& F^{ij}W_{ij,ll} + F^{ij}( - 2 u_{li}u_{lj} -2 u_j (-2 u_l u_{li}
    - \frac{\eta_i}{\eta} K - a K \gd_{in})\\
    & &  + (|\hess u|^2 + u_k (-2 u_l u_{lk} - \frac{\eta_k}{\eta} K - a K \gd_{kn}))
   g_{ij}) - C \sum_i F^{ii} (1 + |\hess
  u|).
  \end{eqnarray*}
 By (\ref{i:grad3}) and the conditions on $\eta,$ we have
   $$
   I \geq F^{ij}W_{ij,ll} + F^{ij}( - 2 u_{li}u_{lj} + 4 u_j u_l u_{li}
     +(|\hess u|^2 - 2 u_k u_l u_{lk}) g_{ij})
    - C \sum_i F^{ii}\eta^{- \frac{1}{2}}( 1 + |\hess u|^{\frac{3}{2}}).
   $$
 For II, we use the formula of
 $W_{ij,l}$  to obtain
  \begin{eqnarray*}
  II &\geq& F^{ij} (2 u_{li}u_{lj}+ 2 u_l W_{ij,l} - 4 u_i
   u_{jl} u_l + 2 u_k u_{kl} u_l g_{ij})
   - C \sum_i F^{ii} (1+ |\hess u|^{\frac{1}{2}}).
  \end{eqnarray*}

 Combining I and II together and after canceling out  six terms,
    \begin{eqnarray*}
  F^{ij} K_{ij} &\geq& F^{ij}W_{ij,ll} + F^{ij}|\hess u|^2  g_{ij}
    + 2 F^{ij}  u_l W_{ij,l}
   - C \sum_i F^{ii} \eta^{-\frac{1}{2}}( 1 + |\hess u|^{\frac{3}{2}}).
  \end{eqnarray*}
Now returning to (\ref{i:b'}),  applying $\eta$ on both sides and
by the concavity of $F,$
  \begin{eqnarray*}
 0 &\geq&  \eta^2 \sum_i F^{ii}|\hess u|^2 + \eta^2 (f(x,u))_{ll}
    + 2 \eta^2 u_l (f (x,u))_l
    - C \sum_i F^{ii} ( 1 + \eta^{\frac{3}{2}} |\hess
    u|^{\frac{3}{2}})\\
   &\geq& \sum_i F^{ii} ( \eta^2 |\hess u|^2   - C - C \eta |\hess u|- C \eta^{\frac{3}{2}} |\hess
   u|^{\frac{3}{2}}).
  \end{eqnarray*}
 This gives $(\eta |\hess u|)(x_0) \leq C.$ Hence, for $x \in
\overline{B}^+_{\frac{r}{2}}$ we have $\Delta u$ and  $|\gra u|$ are bounded.

 (3) For the Hessian bounds, it
 follows   that if $\Gamma^+_2 \subset \Gamma,$
 then $|u_{ij}| \leq C \Delta u.$
\end{proof}

  Department of Mathematics, University of California, Berkeley, CA
 \par
 Email address: \textsf{sophie@math.berkeley.edu}

 Current address:
 
 Institute for Advanced Study, Princeton, NJ
  \par
 Email address: \textsf{sophie@math.ias.edu}


\begin{thebibliography}{10}

\bibitem{And01}
Michael~T. Anderson.
\newblock {$L\sp 2$} curvature and volume renormalization of {AHE} metrics on
  4-manifolds.
\newblock {\em Math. Res. Lett.}, 8(1-2):171--188, 2001.

\bibitem{BG94}
Thomas~P. Branson and Peter~B. Gilkey.
\newblock The functional determinant of a four-dimensional boundary value
  problem.
\newblock {\em Trans. Amer. Math. Soc.}, 344(2):479--531, 1994.

\bibitem{CGY02}
Sun-Yung~A. Chang, Matthew~J. Gursky, and Paul Yang.
\newblock An a priori estimate for a fully nonlinear equation on
  four-manifolds.
\newblock {\em J. Anal. Math.}, 87:151--186, 2002.

\bibitem{CGY02a}
Sun-Yung~A. Chang, Matthew~J. Gursky, and Paul~C. Yang.
\newblock An equation of {M}onge-{A}mp\`ere type in conformal geometry, and
  four-manifolds of positive {R}icci curvature.
\newblock {\em Ann. of Math. (2)}, 155(3):709--787, 2002.

\bibitem{CQY04}
Sun-Yung~A. Chang, Jie Qing, and Paul Yang.
\newblock On the topology of conformally compact {E}instein 4-manifolds.
\newblock volume 350 of {\em Contemp. Math.}, pages 49--61. Amer. Math. Soc.,
  2004.

\bibitem{Chen05}
Szu-yu~Sophie Chen.
\newblock Local estimates for some fully nonlinear elliptic equations.
\newblock {\em Int. Math. Res. Not.}, (55):3403--3425, 2005.

\bibitem{Chen05a}
Szu-yu~Sophie Chen.
\newblock Boundary value problems for some fully nonlinear elliptic equations.
\newblock {\em Calc. Var. Partial Differential Equations}, 30(1):1--15, 2007.

\bibitem{doCar}
Manfredo~Perdig{\~a}o do~Carmo.
\newblock {\em Riemannian geometry}.
\newblock Mathematics: Theory \& Applications. Birkh\"auser Boston Inc., 1992.

\bibitem{Es92}
Jos{\'e}~F. Escobar.
\newblock The {Y}amabe problem on manifolds with boundary.
\newblock {\em J. Differential Geom.}, 35(1):21--84, 1992.

\bibitem{Evan82}
Lawrence~C. Evans.
\newblock Classical solutions of fully nonlinear, convex, second-order elliptic
  equations.
\newblock {\em Comm. Pure Appl. Math.}, 35(3):333--363, 1982.

\bibitem{Gar59}
Lars G{\.a}rding.
\newblock An inequality for hyperbolic polynomials.
\newblock {\em J. Math. Mech.}, 8:957--965, 1959.

\bibitem{GeW05}
Yuxin Ge and Guofang Wang.
\newblock On a fully nonlinear {Y}amabe problem.
\newblock {\em Ann. Sci. \'Ecole Norm. Sup. (4)}, 39(4):569--598, 2006.

\bibitem{GT}
David Gilbarg and Neil~S. Trudinger.
\newblock {\em Elliptic partial differential equations of second order}.
\newblock Springer-Verlag, Berlin, second edition, 1983.

\bibitem{Gb05}
Bo~Guan.
\newblock Conformal metrics with prescribed curvature functions on manifolds
  with boundary.
\newblock preprint.

\bibitem{GLW04}
Pengfei Guan, Chang-Shou Lin, and Guofang Wang.
\newblock Application of the method of moving planes to conformally invariant
  equations.
\newblock {\em Math. Z.}, 247(1):1--19, 2004.

\bibitem{GW03a}
Pengfei Guan and Guofang Wang.
\newblock A fully nonlinear conformal flow on locally conformally flat
  manifolds.
\newblock {\em J. Reine Angew. Math.}, 557:219--238, 2003.

\bibitem{GV05}
Matthew~J. Gursky and Jeff~A. Viaclovsky.
\newblock Prescribing symmetric functions of eigenvalues of {S}chouten tensor.
\newblock to apear in {A}nn. of {M}ath.

\bibitem{GV03}
Matthew~J. Gursky and Jeff~A. Viaclovsky.
\newblock A fully nonlinear equation on four-manifolds with positive scalar
  curvature.
\newblock {\em J. Differential Geom.}, 63(1):131--154, 2003.

\bibitem{GV04}
Matthew~J. Gursky and Jeff~A. Viaclovsky.
\newblock Volume comparison and the {$\sigma\sb k$}-{Y}amabe problem.
\newblock {\em Adv. Math.}, 187(2):447--487, 2004.

\bibitem{Kry83}
N.~V. Krylov.
\newblock Boundedly inhomogeneous elliptic and parabolic equations in a domain.
\newblock {\em Izv. Akad. Nauk SSSR Ser. Mat.}, 47(1):75--108, 1983.

\bibitem{LL03}
Aobing Li and Yanyan Li.
\newblock On some conformally invariant fully nonlinear equations.
\newblock {\em Comm. Pure Appl. Math.}, 56(10):1416--1464, 2003.

\bibitem{LieT86}
Gary~M. Lieberman and Neil~S. Trudinger.
\newblock Nonlinear oblique boundary value problems for nonlinear elliptic
  equations.
\newblock {\em Trans. Amer. Math. Soc.}, 295(2):509--546, 1986.

\bibitem{LT86}
P.-L. Lions and N.~S. Trudinger.
\newblock Linear oblique derivative problems for the uniformly elliptic
  {H}amilton-{J}acobi-{B}ellman equation.
\newblock {\em Math. Z.}, 191(1):1--15, 1986.

\bibitem{Q03}
Jie Qing.
\newblock On the rigidity for conformally compact {E}instein manifolds.
\newblock {\em Int. Math. Res. Not.}, (21):1141--1153, 2003.

\bibitem{Reilly}
Robert~C. Reilly.
\newblock On the {H}essian of a function and the curvatures of its graph.
\newblock {\em Michigan Math. J.}, 20:373--383, 1973.

\bibitem{STW07}
Wei-Min Sheng, Neil~S. Trudinger, and Xu-Jia Wang.
\newblock The {Y}amabe problem for higher order curvatures.
\newblock {\em J. Differential Geom.}, 77(3):515--553, 2007.

\bibitem{TrWx05}
Neil~S. Trudinger and Xu-jia Wang.
\newblock On {H}arnack inequalities and singularities of admissible metrics in
  the {Y}amabe problem.
\newblock preprint.

\bibitem{Via00}
Jeff~A. Viaclovsky.
\newblock Conformal geometry, contact geometry, and the calculus of variations.
\newblock {\em Duke Math. J.}, 101(2):283--316, 2000.

\end{thebibliography}
\end{document}